\documentclass[a4paper, 12pt, notitlepage]{amsart}
\usepackage[colorlinks]{hyperref} 
\usepackage{latexsym, amssymb, amscd, amsfonts, amsthm, amsmath, graphicx,
mathabx, mathrsfs, tikz}

\newcommand{\tri}{\operatorname{tri}}
\newcommand{\hex}{\operatorname{hex}}

\newcommand{\fcc}{\operatorname{fcc}}
\newcommand{\Dfour}{\operatorname{D4}}

\newcommand{\meas}{\operatorname{meas}}

\newcommand{\supp}{\operatorname{supp}}

\newcommand{\vol}{\operatorname{vol}}

\newcommand{\mix}{\operatorname{mix}}

\newcommand{\bH}{\mathbb{H}}

\newcommand{\bR}{\mathbb{R}}

\newcommand{\bT}{\mathbb{T}}

\newcommand{\zed}{\mathbb{Z}}

\newcommand{\one}{\mathbf{1}}
\newcommand{\E}{\mathbf{E}}

\newcommand{\Prob}{\mathbf{Prob}}

\newcommand{\sA}{\mathscr{A}}

\newcommand{\sF}{\mathscr{F}}

\newcommand{\sH}{\mathscr{H}}

\newcommand{\sT}{\mathscr{T}}

\newcommand{\sI}{\mathscr{I}}

\newcommand{\ueps}{\underline{\epsilon}}

\newcommand{\ua}{\underline{a}}

\newcommand{\un}{\underline{n}}
\newcommand{\ut}{\underline{t}}
\newcommand{\ux}{\underline{x}}

\newcommand{\sP}{{\mathscr{P}}}

\newcommand{\fS}{{\mathfrak{S}}}

\newtheorem{theorem}{Theorem}
\newtheorem*{theorem*}{Theorem}

\newtheorem{lemma}[theorem]{Lemma}

\theoremstyle{remark}

\title[Sandpiles]{The spectrum of the abelian sandpile model}
\author{Robert Hough}
\address[Robert Hough]{Department of Mathematics, Stony Brook University, Stony Brook,
NY, 11794}
\email{robert.hough@stonybrook.edu}

\author{Hyojeong Son}
\address[Hyojeong Son]{Department of Mathematics, Stony Brook University, Stony Brook,
NY, 11794}
\email{hyojeong.son@stonybrook.edu}


\thanks{This material is based upon work supported by the National Science
Foundation under agreements No.\ DMS-1712682 and DMS-1802336. Any opinions, findings and
conclusions or recommendations expressed in this material are those of the
authors and do not necessarily reflect the views of the National Science
Foundation.}

\thanks{Hyojeong Son was supported by a fellowship from the Stony Brook Summer Research Fund.}

\begin{document}

\begin{abstract}
In their previous work, the authors studied the abelian sandpile model on graphs constructed from a growing piece of a plane or space tiling, given periodic or open boundary conditions, and identified \emph{spectral factors} which govern the asymptotic spectral gap and asymptotic mixing time.  This article gives a general method of determining the spectral factor either computationally or asymptotically and performs the determination in specific examples.  
\end{abstract}

\maketitle

\section{Introduction} When considering Markovian dynamics in a  system, important quantities in describing the behavior are the spectral gap, or difference between the largest and second largest eigenvalue of the transition kernel, and the convergence profile to equilibrium including the mixing time and transition window between approximately non-uniform and stationarity.  A central topic in the mixing of large systems is the \emph{cut-off phenomenon}, in which, as the system grows, the transition period to equilibrium is on an asymptotically shorter time scale than the mixing time \cite{D96}.

Sandpile dynamics on a graph, which are Markovian, are an important model of self-organized criticality, which have been studied extensively in the statistical physics literature since their introduction by Bak, Tang and Wiesenfeld \cite{BTW88}, see \cite{D89}, \cite{DM92}, \cite{BIP93}, \cite{P94}, \cite{I94}, \cite{DFF03}, \cite{LH02},  \cite{PR05}, \cite{JPR06}, \cite{SV09}, \cite{LP10}, \cite{DS10}, \cite{FLW10},  \cite{G16}, \cite{KW16}, \cite{NOT17}, and \cite{LPS16}, \cite{PS13}.  The article \cite{KS16} provides an accessible introduction, and computes several sandpile statistics for varying graph geometries.    In \cite{HJL17} and \cite{HS19} the authors evaluated the spectral gap, asymptotic mixing time and proved a cut-off phenomenon in sandpile dynamics on a growing piece of a plane or space tiling given periodic or open boundary conditions.  In particular, in the article \cite{HS19}  \emph{spectral factors} related to the harmonic modulo 1 functions on the tiling are identified, and these factors are demonstrated to control the spectral gap and asymptotic mixing time of the dynamics.  The purpose of this article is to describe a general method of calculating the spectral factors numerically, and to perform this calculation in specific examples.  A consequence of these calculations is that, while in two dimensions, the tilings considered have the same asymptotic mixing time to top order with either open or periodic boundary, for the $\Dfour$ lattice in four dimensions, the asymptotic mixing time with an open boundary is longer, and is controlled by the configuration of the sandpile near its 3 dimensional boundary.  Also, it is demonstrated that for all $d$ sufficiently large, for the $\zed^d$ lattice the asymptotic mixing time with periodic and open boundary is the same to top order.

\subsection{Precise statement of results}
A $d$ dimensional periodic plane or space tiling is a connected graph $\sT =(V,E)$ embedded in $\bR^d$ which is periodic in a $d$ dimensional lattice $\Lambda$. Denote $\Delta$  the graph Laplacian,
\[
 \Delta f(v) = \deg(v) f(v) - \sum_{(v,w) \in E}f(w).
\]
Given a function $f\in \ell^2(\sT)$, say that $f$ is \emph{harmonic modulo 1} if $\Delta f \equiv 0 \bmod 1$ and denote the set of such functions $\sH^2(\sT)$. Let $C^1(\sT)$ denote the set of integer valued functions on $\sT$ which are finitely supported and have sum 0.
In \cite{HS19} the \emph{spectral parameter} of a tiling is defined to be
\begin{equation}
 \gamma = \inf \left\{ \sum_{x \in \sT} 1 - \cos(2\pi \xi_x):\Delta \xi \in C^1(\sT), \xi \not \equiv 0 \bmod 1\right\}.
\end{equation}
This parameter is shown to govern the asymptotic spectral gap of sandpile dynamics for graphs with periodic boundary condition, and governs the asymptotic mixing time for the same graphs in dimensions at most 4. 

Our first result computes the spectral factor for the triangular ($\tri$) and honeycomb ($\hex$) tilings in two dimensions and the face centered cubic ($\fcc$) tiling in three dimensions.

\begin{theorem}\label{spectral_gap_calc_theorem}
  The triangular, honeycomb, 
 and face centered cubic tilings have periodic boundary spectral parameters\footnote{The digit in parenthesis indicates the last significant digit.}
 \begin{align*}
 \gamma_{\tri} &= 1.69416(6)\\
 \notag \gamma_{\hex} &= 5.977657(7)\\
 \notag \gamma_{\fcc} &= 0.3623(9).
 \end{align*}
\end{theorem}

Our remaining results concern spectral factors which govern the mixing time of sandpile dynamics on graphs with open boundary condition.  Define coordinate hyperplanes
\[
 H_{i,j} = \{x \in \bR^d: x_i = j\}, \qquad 1 \leq i \leq d, j \in \zed.
\]
In this article, the results regarding spectral factors concern tilings which, after possibly making a rotation and dilation, have reflection symmetry in the family of hyperplanes $\{H_{i,j}\}$ and don't have edges which cross the symmetry planes.

Given a set $S \subset \{1, 2, ..., d\}$, let $\fS_S$ be the group generated by reflections in the hyperplanes $\{H_{j,0}, j \in S\}$ and let $\sA_S(\sT)$ be functions which are anti-symmetric under reflection in each plane $H_{j,0}, j \in S$.
Let $\sH_S^2(\sT)$ denote those $\ell^2$ harmonic modulo 1 functions in $\sA_S(\sT)$.  Again, for $0 \leq i < d$ define the \emph{spectral parameters}\footnote{Note that, in the definition of $\sH_S^2(\sT)$, $\Delta \xi$ is not required to be in $C^1(\sT)$, so that the definitions of $\gamma$ and $\gamma_0$ differ, although the two notions are shown in \cite{HS19} to coincide in dimensions at most 4.}
\begin{equation}
 \gamma_i =  \inf_{\substack{S \subset \{1, 2, ..., d\}\\ |S| = i}} \inf_{\substack{\xi \in \sH_S^2(\sT)\\ \xi \not \equiv 0 \bmod 1}} \sum_{x \in \sT/\fS_S} 1- \cos(2\pi \xi_x).
\end{equation}
In dimension $d \geq 2$ define the $j$th \emph{spectral factor} 
\begin{equation}
 \Gamma_j = \frac{d-j}{\gamma_j}
\end{equation}
and $\Gamma = \max_j \Gamma_j$.
In \cite{HS19}, the following theorem is proved determining the asymptotic mixing time in terms of the spectral factor. Write $\bT_m = \sT/m\Lambda$ for the periodic tiling graph and let $\sT_m$ be the graph formed by giving a fundamental domain for \[\sT/\{ H_{i,mj}, 1 \leq i \leq d, j \in \zed\}\] an open boundary condition.
\begin{theorem*}
For a fixed tiling $\sT$ in $\bR^d$, sandpiles started from a recurrent state on $\bT_m$ have asymptotic total variation mixing time
\begin{equation}
 t_{\mix}(\bT_m) \sim \frac{\Gamma_0}{2} |\bT_m| \log m
\end{equation}
with a cut-off phenomenon as $m \to \infty$.

If the tiling $\sT$ satisfies the reflection condition then sandpile dynamics started from a recurrent configuration on $\sT_m$ have total variation mixing time
\begin{equation}
 t_{\mix}(\sT_m) \sim \frac{\Gamma}{2} |\sT_m| \log m
\end{equation}
with a cut-off phenomenon as $m \to \infty$. 
\end{theorem*}
If $\Gamma = \Gamma_0$  we say that the \emph{bulk} or \emph{top dimensional behavior}  controls the total variation mixing time, and otherwise that the \emph{boundary behavior} controls the total variation mixing time.  In \cite{HS19} it is shown than for 2 dimensional tilings satisfying a reflection condition, the bulk behavior controls the mixing time.

The $\Dfour$ lattice has vertices $\zed^4 \cup \zed^4 + (\frac{1}{2}, \frac{1}{2}, \frac{1}{2}, \frac{1}{2})$ and 24 nearest neighbors of 0
\begin{equation}
U_4 = \{ \pm e_1, \pm e_2, \pm e_3 , \pm e_4\} \cup \left\{\frac{1}{2} (\epsilon_1, \epsilon_2, \epsilon_3, \epsilon_4), \epsilon_i \in \{\pm 1\}\right\},
\end{equation}
which have unit Euclidean length. The elements of the $\Dfour$ lattice are frequently identified with the `Hurwitz quaternion algebra' in which $U_4$ is the group of units.
 Let 
\begin{align*}v_1=(1,1,0,0),\; v_2=(1,-1,0,0),\; v_3 = (0,0,1,1),\; v_4 = (0,0,1,-1),
\end{align*}
 and define hyperplanes
\begin{equation*}
 \sP_j = \{x \in \bR^4: \langle x, v_j\rangle =0\}.
\end{equation*}
The $\Dfour$ lattice has reflection symmetry in the  family of hyperplanes
\begin{equation}
 \sF_{\Dfour} = \{ n v_j + \sP_j: j \in \{1, 2, 3, 4\}, n \in \zed\},
\end{equation}
which can be dilated and rotated to correspond with $\{H_{i,j}\}$.
Our next result determines the boundary spectral parameters and spectral factors for the $\Dfour$ lattice.

\begin{theorem}\label{D4_theorem}
The spectral parameters of the $\Dfour$ lattice with reflection planes $\sF_{\Dfour}$ and open boundary condition are ($\vartheta$ denotes a parameter bounded by 1 in size)
\begin{align*}
\gamma_{\Dfour, 0} &= 0.075554+ \vartheta 0.00024, \\
\gamma_{\Dfour, 1} &= 0.0440957 +\vartheta 0.00017, \\
\gamma_{\Dfour, 2} &= 0.0389569 +\vartheta 0.00013, \\
\gamma_{\Dfour, 3} &= 0.036873324 +\vartheta 0.00012, \\
\gamma_{\Dfour, 4} &= 0.0357604+ \vartheta 0.00011.
\end{align*}
The spectral factors are given by
\begin{align*}
 \Gamma_{\Dfour,0}  &= 52.9428 + \vartheta 0.17,\\
 \Gamma_{\Dfour,1}  &= 68.03486+ \vartheta 0.27,\\
 \Gamma_{\Dfour,2}  &=  51.3393 + \vartheta 0.17,\\
 \Gamma_{\Dfour,3}  &= 27.1201 + \vartheta 0.084.
\end{align*}
In particular, the total variation mixing time of the dynamics on the $\Dfour$ lattice is dominated by the three dimensional boundary behavior.
\end{theorem}

Our final result determines asymptotically the spectral parameters and spectral factors for $\zed^d$ with coordinate hyperplanes as reflecting hyperplanes. 
\begin{theorem}\label{asymptotic_theorem}
 As $d \to \infty$, the spectral parameter of the $\zed^d$ lattice with periodic boundary condition is
  \begin{equation}
  \gamma_{\zed^d} = \frac{\pi^2}{d^2}\left(1 + \frac{1}{2d} + O\left(d^{-2}\right)\right)
 \end{equation}
 and the parameters with open boundary condition are
\begin{equation}
 \gamma_{\zed^d,j} = \frac{\pi^2}{2d^2} \left(1 + \frac{3}{2d} + O_j\left(d^{-2}\right)\right)
\end{equation}
and, uniformly in $j$,
\begin{equation}
 \gamma_{\zed^d, j} \geq \frac{\pi^2}{2d^2 + d}.
\end{equation}
For each fixed $j$,
\begin{equation}
 \Gamma_j = \frac{2d^3 -(2j+3)d^2 + O_j(d)}{\pi^2}.
\end{equation}
In particular, for all $d$ sufficiently large, the total variation mixing time on $\zed^d$ is dominated by the bulk behavior and $\Gamma = \frac{2d^3}{\pi^2}\left(1 - \frac{3}{2d} + O\left(d^{-2}\right)\right)$.
\end{theorem}
Note that, for all $d$ sufficiently large,  $\gamma_{\zed^d} \neq \gamma_{\zed^d,0}$, so that, in the periodic case, the constant $\gamma_{\zed^d}$ which determines the asymptotic spectral gap is not related to the spectral factor $\Gamma_0$ which controls the asymptotic mixing time.

\subsection{Discussion of method}
The harmonic modulo 1 functions considered in this article are evaluated only as functions on $\bR/\zed$, and hence may be assigned values in $\left(-\frac{1}{2}, \frac{1}{2}\right]$. On this interval there are constants $C_1, C_2>0$ such that $C_1 x^2 \leq 1-\cos(2\pi x) \leq C_2 x^2$.  In particular, each $\xi$ considered in the definitions of the spectral factors may be treated as a function in $\ell^2(\sT)$.  

Rather than work with $\xi$, it is more convenient to work with its prevector $\nu = \Delta \xi$, which is integer valued, and hence behaves discretely. The function $\xi$ is recovered from $\nu$ by convolution with the Green's function $g$ on $\sT$, $\xi = g*\nu$.  Since $\Delta$ is bounded from $\ell^2 \to \ell^2$, only prevectors with bounded $\ell^1$ norm need be considered, and in fact, the arguments of \cite{HS19} reduce the determination of the spectral factors to within a prescribed tolerance to a finite calculation.

Given a prevector $\nu$ and a set $S\subset \sT$, the value
\[
 f_S(\xi) = \sum_{x \in S}1-\cos(2\pi \xi_x)
\]
may be estimated from below by constrained minimization programs.  Since the map $\Delta \xi = \nu$ is linear in $\xi$, the constraints are linear.  The objective function $f_S$ is not convex, but $1 - \cos(2\pi \xi_x)$ is convex in the critical region $\left[-\frac{1}{4}, \frac{1}{4}\right]$ and may be approximated piecewise linearly from below outside this region.  Enforcing the constraint $\Delta \xi = \nu$ at only finitely many vertices gives a rapid method of obtaining a lower bound for the value of each $\xi$.  Since the linear constraints involve only the neighbors of the vertex at which the constraint is applied, groups of vertices which are two-separated may be treated additively.  This reduces to a connected component analysis of the prevector $\nu$.  Boundedly many configurations are found to have a sufficiently small value, and then all ways of gluing together these candidates are considered.

To calculate the value of $f(\xi)$, a Fourier representation for the Green's function is used.  A general recipe for giving this Fourier representation for the Green's function of any tiling is given in \cite{HS19}, and this recipe is used in the specific examples considered here.

\section{The Green's function of a tiling}\label{Greens_function_section} 
Throughout, $\sT \subset \bR^d$ is a tiling, which is periodic in a lattice $\Lambda < \bR^d$.  We assume that $0 \in \sT$.
Given a tiling $\sT$ and a vertex $v$, the Green's function of $\sT$ satisfies $\Delta g_v = \delta_v$, where $\delta_v$ is the Kronecker delta function at $v$.  The purpose of this Section is to give a more complete description of the tilings considered, and to develop their Green's functions.

Indicate random walk started from $v$ in $\sT$ by
$Y_{v,0} = v, Y_{v,n+1} = P \cdot Y_{v,n}$. A \emph{stopping time} adapted to the random walk is a random variable $N$ taking values in $\zed_{\geq 0} \cup \{\infty\}$ such that the event $\{N=n\}$ is measurable in the sigma algebra $\sigma(\{Y_{v,0}, Y_{v,1}, ..., Y_{v,n}\})$.  Let $T_v$ be a stopping time for simple random walk started at $v$ in $\sT$ and stopped at the first positive time that it reaches $\Lambda$.  This is the same stopping time as the first positive visit to 0 on the finite state Markov chain given by random walk on $\sT/\Lambda$, and hence $\Prob(T_v > n) \ll e^{-cn}$ for some constant $c > 0$.  Let $\E[T_0] = \alpha > 0$, see e.g. \cite{LPW17} for an introduction to finite state Markov chains and stopping times.

In \cite{HS19} function spaces are defined on $\sT$,
\begin{align*}
 C^0(\sT) &= \left\{f: \sT \to \zed, \|f\|_1 < \infty\right\}, \\
 C^1(\sT) &= \left\{f \in C^0(\sT), \sum_{x \in \sT} f(x) = 0\right\},\\
 C^2(\sT) &= \left\{f \in C^1(\sT), \sum_{x \in \sT}f(x) \E[Y_{x, T_x}]=0\right\}.
\end{align*}
The convolution of the Green's function $g$ on $\sT$ with a function $\eta$ of bounded support in $\sT$ is defined to be 
\[
 g_\eta = \sum_{v \in \sT} \eta(v) g_v.
\]
There it is shown that for $\eta \in C^0(\sT)$, $g_\eta \in \ell^2(\sT)$ if and only if $\eta \in C^\rho(\sT)$ for \[\rho = \left\{\begin{array}{lll}2 && d = 2\\ 1 && d=3,4\\ 0 && d \geq 5 \end{array}\right..\]
Also, a characterization of the spectral parameters is given. Let $\sI = \{\Delta \eta: \eta \in C^0(\sT)\}$ and for $\xi : \sT \to \bR/\zed$,
\[
 f(\xi) = \sum_{x \in \sT} 1-\cos(2\pi \xi_x).
\]
\begin{lemma}
 The spectral parameter $\gamma$ has characterization, in dimension 2,
 \begin{equation*}
  \gamma = \{f(g*\nu): \nu \in C^2(\sT) \setminus \sI\},
 \end{equation*}
and in dimension at least 3,
\begin{equation*}
 \gamma = \{f(g*\nu): \nu \in C^1(\sT) \setminus \sI\}.
\end{equation*}
The parameter $\gamma_j$ has characterization,
\begin{equation*}
 \gamma_j = \inf_{S \subset \{1, 2, ..., d\}, |S| = j}\{f(g*\nu): \nu \in C^{\rho}(\sT) \cap \sA_S(\sT) \setminus \sI\}.
\end{equation*}

\end{lemma}
In this article, the above lemma is used to describe the minimization of the spectral parameter as a search problem over integer valued vectors which are thus discretely distributed.

Let $\mu$ be the probability distribution of $Y_{0,T_0}$ on $\Lambda$.  
The following evaluation of the Green's function of a tiling is given in \cite{HS19}.
\begin{lemma}
 In dimension 2, for $x \in \Lambda$,
 \begin{equation}
  g_0(x) =  \sum_{n=0}^\infty \frac{\mu^{*n}(x)}{\deg x} - \frac{\mu^{*n}(0)}{\deg 0},
 \end{equation}
while in dimension $\geq 3$, 
\begin{equation}
g_0(x) =  \sum_{n=0}^\infty \frac{\mu^{*n}(x)}{\deg x}
\end{equation}
and both sums converge. 
For $x \not \in \Lambda$,
\begin{equation}
 g_0(x) = \E[g_0(Y_{x, T_x})].
\end{equation}
 For $v \not \in \Lambda$, 
 \begin{equation}
  g_v(x) = \frac{1}{\deg x} \E\left[\sum_{j=0}^{T_v-1} \one(Y_{v,j}=x) \right] + \E\left[g_{Y_{v, T_v}}(x)\right].
 \end{equation}

In dimension 2 $g_v(x) \ll 1 + \log (2 + d(v,x))$ and in dimension $n > 2$, $g_v(x) \ll \frac{1}{(1 +d(v,x))^{n-2}}$. If $j \geq 1$ and $\eta \in C^j(\sT)$, 
\[
 g_\eta(x) \ll \frac{1}{1 + d(x,0)^{j + d-2}}
\]
as $d(x,0) \to \infty$.
\end{lemma}

Using this lemma the following explicit evaluations are obtained for several lattice tilings.  These are used for numerical computations.

\subsection{Triangular lattice} This is a lattice tiling, so the Green's function may be calculated without appealing to the stopping time argument above.  Let $v_1 = (1,0)$ and $v_2 = \left(\frac{1}{2}, \frac{\sqrt{3}}{2}\right)$.
\begin{figure}\centering
 \begin{tikzpicture}[node distance = 1.3cm, auto, place/.style = {circle, 
 thick, draw=blue!75, fill=blue!20}]
  \draw (-1, .86602540378)--(2.5, .86602540378);
  \draw (-1,0)--(2.5,0);
  \draw (-1, -.86602540378)--(2.5,-.86602540378);
  \draw (-1, 1.73205080757)--(2.5, 1.73205080757);
  \draw (-0.5, -.86602540378)--(1, 1.73205080757);
  \draw (-1, 0)--(0, 1.73205080757);
  \draw (.5, -.86602540378)--(2, 1.73205080757);
  \draw (2.5, -.86602540378)--(1, 1.73205080757);
  \draw (1.5, -.86602540378)--(0, 1.73205080757);
  \draw (.5, -.86602540378)--(-1, 1.73205080757);
  \draw (-.5, -.86602540378)--(-1, 0);
   \draw (1.5, -.86602540378)--(2.5, .86602540378);
   \draw (2, 1.73205080757)--(2.5, .86602540378);
   \draw [ultra thick, ->] (0,0)--(1,0) node[midway, below] {$v_1$};
   \draw [ultra thick, ->] (0,0)--(.5, .86602540378) node[midway, right] {$v_2$};
   \end{tikzpicture}
   \caption{The triangular lattice is spanned by vectors $v_1, v_2$.}\label{fig:triangular_lattice}
\end{figure}
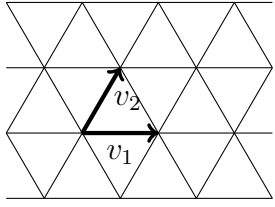
The lattice points take the form $n_1 v_1 + n_2 v_2$.  The lattice graph is regular of degree 6 and the nearest neighbors to 0 are $\{\pm v_1, \pm v_2, \pm (v_1 - v_2)\}$.
Let $\mu$ be the measure
\begin{equation}
 \mu = \frac{1}{6}\left(\delta_{v_1} + \delta_{-v_1} + \delta_{v_2} + \delta_{-v_2} + \delta_{v_1-v_2} + \delta_{v_2-v_1}\right).
\end{equation}
The Green's function from 0 is
\begin{equation}
 g_0(n_1 v_1 + n_2 v_2) = \frac{1}{6}\sum_{n=0}^\infty \mu^{*n}(n_1v_1 +n_2v_2) - \mu^{*n}(0).
\end{equation}
This can be obtained via inverse Fourier transform by
\begin{equation}
 g_0(n_1 v_1 + n_2 v_2) = \frac{1}{6}\int_{\bR^2/\zed^2} \frac{e(n_1 x_1 + n_2x_2)-1}{1 - \frac{1}{3}\left(c(x_1) + c(x_2) + c(x_1-x_2) \right)} dx_1 dx_2.
\end{equation}

\subsection{Honeycomb tiling} This can be constructed from the triangular lattice as follows.  Let $v = \frac{1}{3} (v_1 +v_2)$, which is the centroid of the equilateral triangle with vertices at $\{0, v_1, v_2\}$. 
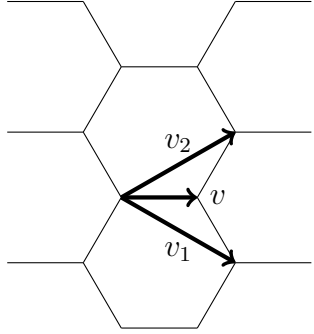
\begin{figure}\centering
  \begin{tikzpicture}[node distance = 1.3cm, auto, place/.style = {circle, 
  thick, draw=blue!75, fill=blue!20}]
  \draw (0, 0)--(1, 0);
  \draw (0, 0)--(-.5, .86602540378);
  \draw (-.5, .86602540378)--(0, 1.73205080757);
  \draw (0, 1.73205080757)--(1, 1.73205080757);
  \draw (1, 1.73205080757)--(1.5, .86602540378);
  \draw (1.5, .86602540378)--(1,0);
  \draw (0,0)--(-.5, -.86602540378);
  \draw (1,0)--(1.5, -.86602540378);
  \draw (-1.5, .86602540378)--(-.5, .86602540378);
  \draw (1.5, .86602540378)--(2.5, .86602540378);
  \draw (0, 1.73205080757)--(-.5, 3*.86602540378);
  \draw (1, 1.73205080757)--(1.5, 3*.86602540378);
  \draw (1.5, 3*.86602540378) --(2.5, 3*.86602540378);
  \draw (-.5, 3*.86602540378)--(-1.5, 3*.86602540378);
  \draw (-.5, -.86602540378)--(-1.5, -.86602540378);
  \draw (1.5, -.86602540378)--(2.5, -.86602540378);
  \draw (1.5, -.86602540378)--(1, -2*.86602540378);
  \draw (-.5, -.86602540378)--(0, -2*.86602540378);
  \draw (0, -2*.86602540378)--(1, -2*.86602540378);
     \draw [ultra thick, ->] (0,0)--(1,0) node[right] {$v$};
   \draw [ ultra thick, ->] (0,0)--(1.5, .86602540378) node[midway, above] {$v_2$};
   \draw [ ultra thick, ->] (0,0)--(1.5, -.86602540378) node[midway, below] {$v_1$};
 \end{tikzpicture}
   \caption{Coordinates in the honeycomb tiling are given in terms of the basis for the triangular lattice, $v_1, v_2$ and $v = \frac{1}{3}(v_1 + v_2)$.}\label{fig:honeycomb_tiling}
 \end{figure}
 
The vertices in the tiling have the form $n_1v_1 + n_2v_2$ and $n_1 v_1 + n_2v_2 + v$ with $n_1, n_2 \in \zed$.  This is a 3-regular graph.  The neighbors of a point $n_1 v_1 + n_2 v_2$ are given by $n_1 v_1 + n_2 v_2 + \{v, -v_1 + v, -v_2 +v\}$.  The neighbors of a point $n_1 v_1 + n_2 v_2 + v$ are $n_1 v_1 + n_2 v_2 + v + \{-v, -v + v_1, -v + v_2 \}$. The tiling has reflection symmetry in the lines in the directions of $v, -v_1 + v, -v_2 + v$ and their translates in the triangular lattice.  Random walk started from 0 stops always on the triangular lattice in two steps, so the stopped measure is
\begin{equation}
 \mu = \frac{1}{3}\delta_0 + \frac{1}{9} \left(\delta_{v_1} + \delta_{-v_1} + \delta_{v_2} + \delta_{-v_2} + \delta_{v_1-v_2} + \delta_{v_2-v_1}\right).
\end{equation}
The Green's function started from 0 is given on the triangular lattice by
\begin{equation}
 g_0(n_1 v_1 + n_2 v_2) = \frac{1}{3} \sum_{n=0}^\infty \mu^{*n}(n_1v_1 + n_2v_2) - \mu^{*n}(0),
\end{equation}
which has the integral representation
\begin{equation}
 g_0(n_1v_1 + n_2v_2) = \frac{1}{3} \int_{\bR^2/\zed^2} \frac{e(n_1 x_1 + n_2 x_2)-1}{\frac{2}{3} - \frac{2}{9}\left(c(x_1) + c(x_2) + c(x_1-x_2) \right)}dx_1 dx_2.
\end{equation}
By harmonicity, 
\begin{align*}
 g_0(n_1v_1 + n_2 v_2 + v) = \frac{1}{3}(&g_0(n_1v_1 + n_2v_2) + g_0((n_1 + 1)v_1 + n_2 v_2)\\& + g_0(n_1 v_1 + (n_2+1)v_2) ).
\end{align*}
By symmetry,
\begin{equation}
 g_v(n_1v_1 + n_2 v_2 + v) = g_0(n_1 v_1 + n_2 v_2).
\end{equation}
Again by harmonicity,
\begin{align*}
 g_v(n_1v_1 + n_2v_2) = \frac{1}{3}(& g_0(n_1v_1 + n_2 v_2) + g_0((n_1 -1)v_1 + n_2v_2)\\
 &+ g_0(n_1v_1 + (n_2-1)v_2)).
\end{align*}

\subsection{Face centered cubic lattice} This is a lattice tiling in $\bR^3$ generated by vectors 
\begin{equation}
 v_1 = \left(1,0,0 \right), \qquad v_2 = \left(\frac{1}{2}, \frac{\sqrt{3}}{2},0 \right), \qquad v_3 = \left(\frac{1}{2}, \frac{\sqrt{3}}{6}, \frac{\sqrt{6}}{3} \right),
\end{equation}
which are the vertices of a regular tetrahedron.  The tiling graph is regular of degree 12.  The neighbors of 0 are
\begin{equation}
 \{ \pm v_1, \pm v_2, \pm v_3, \pm(v_1-v_2), \pm (v_1-v_3), \pm (v_2-v_3)\}.
\end{equation}

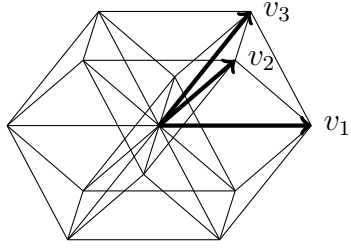
\begin{figure}\centering
\begin{tikzpicture}[scale = 2, node distance = 1.3cm, auto, place/.style = {circle, 
thick, draw=blue!75, fill=blue!20}]
\draw (0,0)--(1,0);
\draw (0,0)--(0.5,0.433012701892219);
\draw (0,0)--(0.602062072615966,0.756710002993201);
\draw (0,0)--(-1,0);
\draw (0,0)--(-0.5,-0.433012701892219);
\draw (0,0)--(-0.602062072615966,-0.756710002993201);
\draw (0,0)--(0.5,-0.433012701892219);
\draw (0,0)--(0.397937927384034,-0.756710002993201);
\draw (0,0)--(-0.102062072615966,-0.323697301100982);
\draw (0,0)--(-0.5,0.433012701892219);
\draw (0,0)--(-0.397937927384034,0.756710002993201);
\draw (0,0)--(0.102062072615966,0.323697301100982);
\draw (1,0)--(0.5,0.433012701892219);
\draw (1,0)--(0.602062072615966,0.756710002993201);
\draw (1,0)--(0.5,-0.433012701892219);
\draw (1,0)--(0.397937927384034,-0.756710002993201);
\draw (0.5,0.433012701892219)--(0.602062072615966,0.756710002993201);
\draw (0.5,0.433012701892219)--(-0.5,0.433012701892219);
\draw (0.5,0.433012701892219)--(-0.102062072615966,-0.323697301100982);
\draw (0.602062072615966,0.756710002993201)--(0.102062072615966,0.323697301100982);
\draw (0.602062072615966,0.756710002993201)--(-0.397937927384034,0.756710002993201);
\draw (0.102062072615966,0.323697301100982)--(-0.397937927384034,0.756710002993201);
\draw (-1,0)--(-0.5,-0.433012701892219);
\draw (-1,0)--(-0.602062072615966,-0.756710002993201);
\draw (-1,0)--(-0.5,0.433012701892219);
\draw (-1,0)--(-0.397937927384034,0.756710002993201);
\draw (-0.5,-0.433012701892219)--(-0.602062072615966,-0.756710002993201);
\draw (-0.5,-0.433012701892219)--(0.5,-0.433012701892219);
\draw (-0.5,-0.433012701892219)--(0.102062072615966,0.323697301100982);
\draw (-0.602062072615966,-0.756710002993201)--(-0.102062072615966,-0.323697301100982);
\draw (-0.602062072615966,-0.756710002993201)--(0.397937927384034,-0.756710002993201);
\draw (-0.102062072615966,-0.323697301100982)--(0.397937927384034,-0.756710002993201);
\draw (0.5,-0.433012701892219)--(0.397937927384034,-0.756710002993201);
\draw (-0.5,0.433012701892219)--(-0.397937927384034,0.756710002993201);
\draw (0.5,-0.433012701892219)--(0.102062072615966,0.323697301100982);
\draw (-0.5,0.433012701892219)--(-0.102062072615966,-0.323697301100982);
   \draw [ultra thick, ->] (0,0)--(1,0) node[right] {$v_1$};
\draw [ ultra thick, ->] (0,0)--(0.5,0.433012701892219) node[right] {$v_2$};
\draw [ ultra thick, ->] (0,0)--(0.602062072615966,0.756710002993201) node[right] {$v_3$};
\end{tikzpicture}
   \caption{Coordinates in the face centered cubic lattice are given in terms of the vectors $v_1, v_2, v_3$.  Any two of these  span the triangular lattice.}\label{fig:fcc_lattice}
\end{figure}

Let $\mu$ be the measure which is uniform on these points. 
The Green's function started from 0 is given by
\begin{align*}
& g_0(n_1v_1 + n_2v_2 + n_3v_3) = \frac{1}{12} \sum_{n=0}^\infty \mu^{*n}(n_1v_1 +n_2v_2 + n_3v_3).
\end{align*}
The Fourier transform is
\begin{align*}
 &\hat{g}(x_1, x_2, x_3)=\\& \frac{1}{12 -2 (c(x_1) + c(x_2) + c(x_3) + c(x_1-x_2) + c(x_1-x_3) + c(x_2-x_3)) }.
\end{align*}

\subsection{$\Dfour$ lattice} The $\Dfour$ lattice is a lattice in $\bR^4$ which is frequently presented as the integer quaternion ring
\begin{equation}
 \bH(\zed) = \{n_1 + n_2 i + n_3 j + n_4k: \un \in \zed^4\}
\end{equation}
together with the points with odd half integer coordinates, 
\begin{equation}
 \Dfour = \bH(\zed) \cup \left(\bH(\zed) + \frac{1}{2}(1 + i +j+k)\right).
\end{equation}
This is a lattice tiling, which is regular of degree 24 as a graph.  The 24 neighbors of 0 are the units of the corresponding quaternion algebra,
\begin{equation}
U_4= \{\pm 1, \pm i, \pm j, \pm k\} \cup \left\{\frac{1}{2}(\epsilon_1+ \epsilon_2i+ \epsilon_3j+ \epsilon_4 k): \ueps \in \{\pm 1\}^4\right\}. 
\end{equation}
A basis for the lattice is given by $v_1 = 1$, $v_2 = i$, $v_3 = j$, $v_4 = \frac{1}{2}(1 + i +j + k)$. In these coordinates, the neighbors of 0 are
\begin{align*}
 &\{\pm v_1, \pm v_2, \pm v_3, \pm (2v_4 -v_1-v_2-v_3), \pm v_4, \pm (-v_1 + v_4), \pm(-v_2 + v_4), \\&\pm(-v_3+v_4), \pm(-v_1-v_2 + v_4), \pm(-v_1-v_3+v_4),\\& \pm (-v_2-v_3+v_4), \pm(-v_1-v_2-v_3+v_4)\}.
\end{align*}
Let $\mu$ be uniform on the neighbors of 0.  This measure has Fourier transform
\begin{align*}
 \hat{\mu}(x_1,x_2,x_3,&x_4) = \frac{1}{12}( c(x_1)+ c(x_2) + c(x_3) + c(2x_4-x_1-x_2-x_3)\\& + c(x_4) + c(-x_1 + x_4) + c(-x_2 + x_4) + c(-x_3+x_4)\\& + c(-x_1-x_2+x_4) + c(-x_1 -x_3 +x_4) + c(-x_2-x_3+x_4)\\& + c(-x_1-x_2-x_3+x_4)).
\end{align*}
The Green's function is given by
\begin{align*}
 &g_0(n_1v_1+n_2v_2+n_3v_3+n_4v_4) \\&= \frac{1}{24}\int_{\bR^4/\zed^4} \frac{e(n_1x_1+n_2x_2+n_3x_3+n_4x_4)}{1-\hat{\mu}(x_1,x_2,x_3,x_4)}dx_1dx_2dx_3dx_4.
\end{align*}

\subsection{$\zed^d$ lattice} For $d \geq 3$ the lattice $\zed^d$ has Green's function
\begin{equation}
 g_0(\un) = \frac{1}{2d} \int_{\bR^d/\zed^d} \frac{e(\un \cdot \ux)}{1- \frac{1}{d}\left(c(x_1)+\cdots+ c(x_d) \right)} d\ux.
\end{equation}

\section{Optimization problem and computer search}

In this section the spectral parameters are determined by computer search for several tilings. 
Recall that
\begin{equation*}
 \gamma = \inf \left\{ \sum_{x \in \sT} 1 - \cos(2\pi \xi_x):\Delta \xi \in C^1(\sT), \xi \not \equiv 0 \bmod 1\right\},
\end{equation*} and 
\begin{equation*}
 \gamma_i =  \inf_{\substack{S \subset \{1, 2, ..., d\}\\ |S| = i}} \inf_{\substack{\xi \in \sH_S(\sT)\\ \xi \not \equiv 0 \bmod 1}} \sum_{x \in \sT/\fS_S} 1- c( \xi_x).
\end{equation*}

The following arguments index harmonic modulo 1 function $\xi$ with its prevector $\nu = \Delta \xi$, which is simpler as the prevector is integer valued.  This permits an approximate ordering on prevectors in terms of their norm, and the diameter of their support.  The harmonic modulo 1 function is then recovered as $\xi = g * \nu$ as the following lemma demonstrates.
\begin{lemma}
 Let $\xi \in \ell^2(\sT)$ be harmonic modulo 1, and let $\nu = \Delta \xi$ be its prevector.  Then $\xi = g_\nu$.
\end{lemma}
\begin{proof}
 Given $x \in \Lambda$, let $T_x$ denote translation by $x$, and let $\xi^x = \xi - T_x \xi$.  Hence $\nu^x = \Delta(\xi^x) = \nu - T_x \nu$ is in $C^1(\sT)$.  It follows from \cite{HS19} that $g_{\nu^x}(y) \to 0$ as $d(0,y) \to \infty$.  Since 
 \[
  \Delta(\xi^x - g_{\nu^x}) = \nu^x - \nu^x = 0
 \]
and since $\xi^x - g_{\nu^x}$ vanishes at infinity, it follows from the maximum modulus principle that $\xi^x = g_{\nu^x}$.  As $x \to \infty$ for each fixed $y$, $\xi^x(y) \to \xi(y)$ and hence $g_{\nu}$ tends to 0 at infinity.  The argument may now be repeated with $\xi$ and $\nu$ replacing $\xi^x$ and $\nu^x$ to conclude $\xi = g_{\nu}$.
\end{proof}

The next lemma controls cosine sums of $\xi$ in terms of the $\ell^2$ norm.

\begin{lemma}\label{2_norm_lemma}
 Let $S$ be a finite or countable set and let $\xi \in \ell^2(S)$, $\|\xi\|_\infty \leq \frac{1}{2}$.  Define 
 \begin{equation}
  f_S(\xi) = \sum_{x \in S} 1 - c(\xi_x).
 \end{equation}
Let $\alpha > 0$ and assume $\|\xi\|_2^2 \geq \alpha$.  Then
\begin{equation}
  2\pi^2 \alpha \left(1 - \frac{\pi^2}{3} \alpha\right) \leq f_S(\xi) \leq 2\pi^2 \|\xi\|_2^2.
\end{equation}
\end{lemma}

\begin{proof}
The Taylor series approximation for $c(x)$ on $|x| \leq \frac{1}{2}$,
\begin{align*}
 c(x) = 1 - 2 \pi^2 x^2 + \frac{2\pi^4}{3} x^4 - \cdots
\end{align*}
 is an alternating series with decreasing increments  after the term $2\pi^2 x^2$.  Thus $f_S(\xi) \leq 2\pi^2 \|\xi\|_2^2$. 
 Let $0 < \lambda \leq 1$ and let $\xi' = \lambda \xi$ satisfy  $\| \xi'\|_2^2 = \alpha$.  Then $f_S(\xi') \leq f_S(\xi)$.  Furthermore, 
 \begin{align*}
f_S(\xi') &\geq 2\pi^2 \|\xi'\|_2^2 - \frac{2}{3}\pi^4 \|\xi'\|_4^4 \\&\geq 2\pi^2 \alpha - \frac{2}{3} \pi^4 \alpha \|\xi'\|_\infty^2 \\&\geq 2\pi^2 \alpha - \frac{2}{3}\pi^4 \alpha^2.
 \end{align*}
\end{proof}
The following lemma is used to estimate the functionals $f(\xi)$.
\begin{lemma}\label{f_approx_lemma}
 Let $R \subset \sT$ and let $\xi: \sT \to \left(-\frac{1}{2}, \frac{1}{2}\right]$. Let 
 \begin{equation}
\left\|\xi\right\|_{2, R^c}^2 = \sum_{x \in \sT \setminus R} \xi_x^2.  
 \end{equation}
 There is a number $\vartheta$, $|\vartheta| \leq 1$ such that
 \begin{equation}
  f(\xi) = \sum_{x \in R}\left( 1-c(\xi_x)\right) +2\pi^2 \|\xi\|_{2, R^c}^2 - \frac{\pi^4}{3} \|\xi\|_{2,R^c}^4 + \vartheta \frac{\pi^4}{3}\|\xi\|_{2,R^c}^4.
 \end{equation}

\end{lemma}
\begin{proof}
 By Taylor approximation, for $x \in R^c$, \begin{equation*}2\pi^2 \xi_x^2 - \frac{2}{3} \pi^4 \xi_x^4 \leq 1-c(\xi_x) \leq 2\pi^2 \xi_x^2.\end{equation*}  Thus,
 \begin{align*}
  \sum_{x \in R} (1-c(\xi_x))& + 2\pi^2 \|\xi_x\|_{2,R^c}^2 - \frac{2}{3}\pi^4 \|\xi\|_{2, R^c}^4 \\&\leq f(\xi) \leq \sum_{x \in R} (1-c(\xi_x)) + 2\pi^2 \|\xi\|_{2,R^c}^2,
 \end{align*}
from which the claim follows.
\end{proof}
In practice, Lemma \ref{f_approx_lemma} is applied by calculating $\xi_x$ on $R$ from the Fourier integral representations in Section \ref{Greens_function_section} in a neighborhood of 0, and calculating $\|\xi\|_2^2$ by Parseval.  Note that each Fourier integral contains a singularity at 0.  The integrand can be converted to a bounded function of bounded derivatives by switching to spherical coordinates of the appropriate dimension.

The following two optimization programs are used to obtain a lower bound for $f(\xi)$. 
Let $\xi = g* \nu$, $\|\xi\|_\infty \leq \frac{1}{2}$.  Given a set $S \subset \sT$, a lower bound for $f(\xi)$ is obtained as the solution of the optimization program $Q(S, \nu)$,
\begin{align*}
 Q(S, \nu):&\\
 \text{minimize:}& \qquad \sum_{d(w, S) \leq 1} 1-c(x_w)\\
 \text{subject to:}& \qquad \forall u \in S,\; (\deg u) x_u - \sum_{d(w,u)=1} x_w = \nu_u\\
 & \qquad -\frac{1}{2} \leq x_w \leq \frac{1}{2}.
\end{align*}
A lower bound for $Q(S, \nu)$ is the relaxed optimization program with positive constraints $P(S, \nu)$
\begin{align*}
 P(S, \nu):&\\
 \text{minimize:}& \qquad \sum_{d(w, S) \leq 1} 1-c(x_w)\\
 \text{subject to:}& \qquad \forall u \in S,\; (\deg u) x_u + \sum_{d(w,u)=1} x_w \geq \nu_u\\
 & \qquad -\frac{1}{2} \leq x_w \leq \frac{1}{2}.
\end{align*}
Note that the objective function is convex and with non-degenerate Hessian in the interior with the stronger condition $|x_w|\leq \frac{1}{4}$, and hence has a unique local minima there.  In order to estimate $Q(S, \nu)$ and $P(S, \nu)$ numerically, the range $\frac{1}{4} \leq |x_w| \leq \frac{1}{2}$ was split into several equal size intervals and the objective function was approximated piecewise linearly on these, obtaining a lower bound for the minimum.  The minima were compared with the variables constrained to lie in each interval. Denote $P_j(S, \nu)$ and $Q_j(S, \nu)$ the programs in which both $\left[-\frac{1}{2}, -\frac{1}{4}\right]$ and $\left[\frac{1}{4}, \frac{1}{2}\right]$ are split into $j$ equal size intervals, and objective function interpolating linearly between the values of $c(x)$ on the endpoints.  Note that the minimum of $P_j$ and $Q_j$ on each product of intervals is determined as a unique interior minimum or boundary value.
In the examples considered in dimensions 3 and higher, $\|\xi\|_2^2$ was optimized rather than $f(\xi)$, and it was demonstrated that the extremal function is the same.  Programs $Q'(S, \nu)$ and $P'(S, \nu)$ have the same constraints, but have objective function
$\sum_{d(w, S) \leq 1} x_w^2.$  Note that this objective function is convex.

The optimization programs $P, P_j, P', Q, Q_j, Q'$ satisfy the following monotonicity properties.
\begin{lemma}
 The programs $P, P_j, P', Q, Q_j, Q'$ are monotone increasing in the set $S$.  The programs $P, P_j, P'$ are monotone increasing in the prevector $|\nu|$.
\end{lemma}

\begin{proof}
 This follows from constraint relaxation.
\end{proof}

The programs also satisfy the following additivity property.
\begin{lemma}
 Let $B(S) = \{u: d(u, S) \leq 1\}$ be the distance 1 enlargement of $S$.  When $S_1, S_2, ..., S_k$ are some sets in $\sT$ whose distance 1 enlargements $B(S_1), B(S_2), ..., B(S_k)$ are pairwise disjoint, then $\sum_{i=1}^k Q(S_i, \nu) \leq f(\xi)$ and $\sum_{i=1}^k Q'(S_i, \nu) \leq \|\xi\|_2^2$.
\end{lemma}
\begin{proof}
 Since the sets of variables are disjoint, the sum of the optimization programs can be considered to be a single optimization program, which is then satisfied by the optimizing solution $\xi$.  The corresponding values for $\xi$ are thus an upper bound on the optimum.
\end{proof}
Since the remaining programs $P, P', P_j, Q_j$ are relaxations of $Q$ and $Q'$, the additivity property holds for these as well.

A basic estimate for the value of $Q'$ is as follows.

\begin{lemma}\label{2_norm_opt_lemma}
 Let $G = (V,E)$ be a graph and let $v \in V$ of degree at least 2, with a single edge to each of its neighbors and no self-loops. Let $|\nu_v| = 1$.   The optimization problem $Q'(\{v\}, \nu)$ 
has value $\frac{1}{\deg(v)(\deg(v) + 1)}$.
\end{lemma}
\begin{proof}
Assume without loss of generality that $\nu_v = 1$. The constraint is $(\deg v) x_v - \sum_{(v,w) \in E} x_w = 1$ and the objective function is $x_v^2 + \sum_{(v,w) \in E} x_w^2$. Since the claimed value is smaller than the value on the boundary, it may be assumed that the optimum is achieved at an interior point.
By Lagrange multipliers, there is a scalar $\lambda$ such that  $x_v = \lambda \deg v$ and $x_w = -\lambda$ for all $(v,w) \in E$.  Thus $\lambda = \frac{1}{\deg(v)(\deg(v) +1)}$.  The claim follows, since \begin{equation}\sum_{d(v,w) \leq 1} x_w^2 = \lambda^2 \deg(v)(\deg(v)+1).\end{equation}
\end{proof}

In particular, combining this lemma with the additivity property above proves that the extremal prevector has a bounded $\ell^1$ norm.

The strategy of the arguments is now described as follows.  Say two points $x_i, x_t$ in the support of $\nu$ are  2-path connected, or just connected for short, if there is a sequence of points $x_i = x_0, x_1, ..., x_n = x_t$ in the support of $\nu$, such that the graph distance between $x_i$ and $x_{i+1}$ is at most 2. By the additivity lemma, the value of the optimization programs applied with $S_i$ separated connected components of $\supp \nu$ is additive.  Since the value of each optimization program is translation invariant and, for a fixed $\nu$, monotone in $S$, all connected components with $P$ or $Q$ (resp. $P', Q', P_j, Q_j$) value at most a fixed constant can be enumerated  by starting from a base configuration and adding connected points to the set $S$ one at a time.

The configuration $\nu$ must be in $C^\rho$ for $\xi \in \ell^2(\sT)$.  Having enumerated all feasible connected components, the search is completed by considering all methods of gluing together several connected components which produce a $\nu \in C^\rho$.  

\subsection{Issues of precision}
The techniques used in this section consist in the following:
minimization of a convex function in a convex bounded region, which can be certified by calculation of the derivative of the objective function at the optimum found, and integration of function with bounded derivatives over a bounded domain.  Although the integrals involving the characteristic function of a Green's function may have a singularity at 0, this may be removed in each case by switching to spherical coordinates of the correct dimension near the point of singularity.  Thus the numerical results are verifiable to within the claimed precision.  The generating code written in SciPy is available from the authors upon request.

\subsection{Proof of Theorems \ref{spectral_gap_calc_theorem} and \ref{D4_theorem}}
\subsubsection{Triangular lattice case}
Let the triangular lattice be generated by $v_1 = (1,0)$ and $v_2 = \left(\frac{1}{2}, \frac{\sqrt{3}}{2}\right)$.  Let $\xi^* = g* \nu^*$ with $\nu^* = \delta_{0} -\delta_{v_1} - \delta_{v_2} + \delta_{v_1+v_2}$.  The value 
\begin{equation}
 f(\xi^*) = 1.69416(5)
\end{equation}
was estimated by Lemma \ref{f_approx_lemma} with 
\begin{equation}
R = \{n_1 v_1 + n_2 v_2: \max(|n_1|, |n_2|) \leq 10\}.
\end{equation}
   It is to be shown that $\gamma_{\tri} = f(\xi^*)$.
\begin{figure}
 \begin{tikzpicture}[node distance = 1.3cm, auto, place/.style = {circle, 
 thick, draw=blue!75, fill=blue!20}]
  \draw (-1, .86602540378)--(2.5, .86602540378);
  \draw (-1,0)--(2.5,0);
  \draw (-1, -.86602540378)--(2.5,-.86602540378);
  \draw (-1, 1.73205080757)--(2.5, 1.73205080757);
  \draw (-0.5, -.86602540378)--(1, 1.73205080757);
  \draw (-1, 0)--(0, 1.73205080757);
  \draw (.5, -.86602540378)--(2, 1.73205080757);
  \draw (2.5, -.86602540378)--(1, 1.73205080757);
  \draw (1.5, -.86602540378)--(0, 1.73205080757);
  \draw (.5, -.86602540378)--(-1, 1.73205080757);
  \draw (-.5, -.86602540378)--(-1, 0);
   \draw (1.5, -.86602540378)--(2.5, .86602540378);
   \draw (2, 1.73205080757)--(2.5, .86602540378);
   \node[place] at (0,0)  (lower)[label=$1$]  {};
   \node[place] at (1,0)  (right lower) [label=$-1$]{};
   \node[place] at (.5,.86602540378)  (lower)[label=$-1$]  {};
   \node[place] at (1.5,.86602540378)  (right lower) [label=$1$]{};
 \end{tikzpicture}
\caption{The extremal configuration for the triangular lattice.}\label{fig:triangular_lattice_extremal}
 \end{figure}
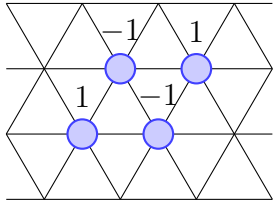
 
First the case of a node of height 3 in the extremal prevector is ruled out.  
 \begin{lemma}
  Suppose $|\nu_0| = 3$.  The optimization program $P(\{0\}, |\nu|)$ has value 2. In particular, $\nu$ does not achieve $\gamma_{\tri}$.
 \end{lemma}
\begin{proof}
 At the optimum, the largest value is $x_0$, since if $x_w$ is larger for some $w$ with $d(w,0)=1$ then the constraint may be improved by swapping $x_0$ and $x_w$.  It follows that $|x_w| \leq \frac{1}{4}$ for $w \neq 0$ since otherwise the claimed bound would be exceeded.  For a fixed $x_0$, the conditioned optimization problem is now convex with a unique local minimum, which by symmetry occurs with all variables equal.  This reduces to minimizing $1-c(x) + 6 \left(1 - c\left(\frac{1}{2}-x\right)\right)$ for $0 \leq x \leq \frac{1}{2}$, which has minimum 2.
 
\end{proof}

 Next the possibility of a prevector with node of height at least 2 is ruled out. 
 \begin{lemma}\label{singleton_lemma}
  If $|\nu_0| = 2$, $P(\{0\}, |\nu|) \geq 1.4322$.  If $|\nu_0| = 1$, $P(\{0\}, |\nu|) \geq 0.44256$.
 \end{lemma}
\begin{proof}
 These values were verified in SciPy.
\end{proof}

 It follows that if the minimizing prevector has a node of height 2, it does not have any non-zero node at distance greater than 2 from the node of height 2, since otherwise the two optimization problems could be applied separately at the two nodes, and the total value would exceed $\gamma_{\tri}$.
 
 Up to rotation, there are two types of nodes at graph distance 2 from 0 in $\sT$, $v_1 + v_2$ and $2v_1$.  A non-zero node at distance two is ruled out by considering the following optimization problems.
 
 \begin{lemma}
  Suppose $|\nu_0| = 2$ and $|\nu_{v_1 + v_2}| = 1$.  Then $P(\{0, v_1 + v_2\}, |\nu|) \geq 1.83$.  If $|\nu_0| = 2$ and $|\nu_{2v_1}| = 1$ then $P(\{0, 2v_1\}, |\nu|) \geq 1.85$.
 \end{lemma}

 \begin{proof}
  These values were verified in SciPy.
 \end{proof}
 Note that $P(S, |\nu|)$ is increasing in $|\nu|$.
The above lemmas prove that if the optimizing prevector $\nu$ has a node of height 2, then any non-zero node in $\nu$ is adjacent to the node of height 2. After translation and multiplying by $\pm 1$, assume $\nu_0 = 2$.   The case in which all six neighbors of 0 are non-zero is ruled out as follows.
\begin{lemma}
 Let $|\nu_0| = 2$ and $|\nu_w| \geq 1$ for each $w$ with $d(w,0) = 1$.  Let $S = \{w: d(w,0) \leq 1\}$.  Then $P(S, |\nu|) \geq 1.9233$.
\end{lemma}

\begin{proof}
 This was verified in SciPy.
\end{proof} 

Similarly, there are not two adjacent nodes of height 2, as the following lemma verifies.
\begin{lemma}
 Suppose that $|\nu_0| = 2$ and $|\nu_{v_1}| = 2$.   Then $P(\{0, v_1\}, |\nu|) \geq 2.3.$
\end{lemma}
\begin{proof}
 This was verified in SciPy.
\end{proof}
Since it is necessary that $\nu \in C^2(\sT)$ for $\xi \in \ell^2(\sT)$, the remaining possible configurations have an even number of non-zero nodes adjacent to 0.  There must be at least 2, and when there are two, the configuration is, up to rotation,   $\nu = -\delta_{-v_1} + 2\delta_0 - \delta_{v_1}$ which has $f(\xi) \geq 2.23$.  No configuration with four non-zero nodes is in $C^2(\sT)$.  This concludes the proof that there is not a node of height 2.

Next decompose the support of $\nu$ into 2-path connected components.  The next stage in the argument reduces to the case of a single connected component.  If there were four or more connected components, Lemma \ref{singleton_lemma} could be applied at a node in each connected component,  which obtains a value at least $4 \times 0.44256> 1.76$.  Hence there are at most 3 connected components, and since $\nu \in C^2(\sT)$, one must contain more than one node. 

\begin{lemma}\label{doubles_lemma}
 If $|\nu_0| =1$ and $|\nu_{v_1}| = 1$ then 
 \begin{equation}\label{adjacent_nodes}
P(\{0, v_1\}, |\nu|) \geq 0.6729.
 \end{equation}
 If $|\nu_0| = 1$ and $|\nu_{v_1 + v_2}| = 1$ then 
 \begin{equation}\label{2_adjacent}
P(\{0, v_1 + v_2\}, |\nu|) \geq 0.8509.  
 \end{equation}
 If $|\nu_0| = 1$ and $|\nu_{2v_1}| = 1$ then 
 \begin{equation}\label{2_adjacent_2}
P(\{0, 2v_1\}, |\nu|) \geq 0.8677.
 \end{equation}
\end{lemma}
\begin{proof}
 These were verified in SciPy.
\end{proof}
It follows that if there are 3 connected components then the only possibility is that one has diameter 1 as in (\ref{adjacent_nodes}) and the other two are singletons, since otherwise the sum of the values of the programs exceeds $\gamma$. To remain in $C^2(\sT)$, the configuration of diameter 1 has two nodes since the total number of nodes is even. 
\begin{lemma}\label{singleton_space_lemma}
 Let $\nu_0 = 1$ and $\nu_w = 0$ for $w$ such that $d(w,0) = 1$.  Let $S = \{w : d(w,0)\leq 1\}$.  Then $Q(\nu, S) \geq 0.9127$. 
\end{lemma}

\begin{proof}
 This was verified in SciPy.
\end{proof}
If there were an optimal configuration with 3 connected components, then the component with two adjacent nodes must have both nodes of equal sign for the configuration to be in $C^2(\sT)$. Thus the two singletons would be placed symmetrically opposite the center of the configuration of size 2 and have the same sign.  Since they are disconnected, they have distance at least 3 from the component of size 2.  It follows that Lemma \ref{singleton_space_lemma} can be applied at each singleton so that the value exceeds $\gamma_{\tri}$.  This eliminates the case of 3 connected components.

Next suppose that there are two connected components.  By applying (\ref{2_adjacent}) and (\ref{2_adjacent_2}) it follows that at least one of the connected components has diameter at most 1.
\begin{lemma}
 Suppose $\nu_0 = \nu_{v_1} = 1$. Then $Q(\{0, v_1\}, \nu) \geq 1.1518$.
\end{lemma}
\begin{proof}
 This was verified in SciPy.
\end{proof}
If one connected component has such a large $Q$ value, then by Lemma \ref{doubles_lemma}, the other component can only be a singleton.  The case of two connected components with one a singleton is deferred to the end of the proof.  Thus consider the case of only connected components of size at least 2 in which adjacent nodes have opposite signs.  It follows that one of the components of diameter 1 has size 2, with adjacent nodes of opposite sign.
\begin{lemma}\label{signed_adjacent_spaced_lemma}
 Let $\nu_0 = 1$, $\nu_{v_1}=-1$.  Let $S = \{w: d(w, \{0, v_1\})\leq 1\}$ and assume $\nu_w = 0$ if $d(w, \{0, v_1\}) = 1$.  Then $Q(S, \nu) \geq 0.971$.
\end{lemma}
\begin{proof}
 This was verified in SciPy.
\end{proof}
Combining Lemma \ref{signed_adjacent_spaced_lemma} with Lemma \ref{doubles_lemma} if one of the connected components has diameter greater than 1, then the component of diameter 1 has distance from it at most 3, hence exactly 3 since the components are not connected.  As in the case of a singleton, this case is deferred to the end of the discussion.  If both components have diameter 1, then to be in $C^2(\sT)$, both have size two and have adjacent nodes of opposite sign.  Applying Lemma \ref{signed_adjacent_spaced_lemma} to each, these are separated by at most distance 4.  This reduces to a finite check, and none of the configurations achieves the optimum.

The argument above  reduces to considering either prevectors with support that are 2 path connected, or prevectors of diameter greater than 1 which are connected at distance 2, together with a second  connected component which is either a singleton or a pair of adjacent nodes of opposing signs.  By combining Lemmas \ref{singleton_space_lemma} and \ref{signed_adjacent_spaced_lemma} with Lemma \ref{doubles_lemma}, it follows that if there is a second connected component it has distance exactly 3 from the component of diameter greater than 1.

The proof is now concluded by computer search.  All connected components $C$ up to translation and symmetry were enumerated, which satisfied one of the following three criteria, $P(C, 1) \leq \gamma_{\tri} = 1.69416(5)$, $P(C, 1) \leq \gamma_{\tri} - 0.44256$, $P(C, 1) \leq \gamma_{\tri} - 0.6729$, with $\nu = 1$ indicating $\nu_x = 1$ for all $x$. The first list consists of all candidate supports which are connected and may give the optimum. By Lemmas \ref{singleton_lemma} and \ref{doubles_lemma}, the latter two lists enumerate configurations which may be paired with a singleton or a pair of adjacent nodes. Since $P(C, 1)$ is increasing in $C$, the enumeration was performed by building configurations from the base $C = \{0\}$ adding neighbors at distance 1 or 2, until the appropriate limit was exceeded.  The first list contains configurations with at most 7 vertices, the second list contains configurations with at most 5 vertices and the third list contains configurations with at most 4 vertices.

Note that a configuration which can appear with adjacent and opposite signed nodes and have a $C^2(\sT)$ assignment of signs must have an even number of nodes.  Also, those of size 2 have already been considered.  Only one configuration on 4 nodes, and no configurations on more nodes had a sufficiently small value of $P(C)$.  The configuration on 4 nodes was, up to symmetries, $\{0, v_1, v_2, v_1 + v_2\}$.  However, there is no assignment of signs which makes this configuration in $C^2(\sT)$ when paired with an adjacent pair of nodes with opposite signs. A connected component with 3 vertices cannot be assigned signs in such a way that a singleton can be added at distance 3 to make a configuration in $C^2(\sT)$, since the distance between the one pair of opposite signed nodes must match the other.  There is a single configuration on 5 nodes with $P$ value less than $\gamma_{\tri} - 0.44256$. There are 4 ways of assigning signs so that a singleton can be added that makes the configuration in $C^2(\sT)$.  Each of these was tested and none give the extremal configuration.  This reduces to the case of connected components.  This finite check was performed in SciPy and obtains $\nu_0$ and $\xi_0$ as claimed.

\subsubsection{Honeycomb tiling case}\label{honeycomb_section}

Let $v_1 = (1,0)$ and $v_2 = \left(\frac{1}{2}, \frac{\sqrt{3}}{2} \right)$ and $v = \frac{1}{3}(v_1 + v_2)$. Thus the points in the honeycomb lattice have the form $n_1 v_1 + n_2 v_2 + n_3 v$ with $n_3 \in \{0,1\}$ and $n_1, n_2 \in \zed$.  The optimal configuration is given by $\xi^* = g*\nu^*$
\begin{equation}
 \nu^* = \delta_0 - \delta_v + \delta_{v_2} - \delta_{-v_1 + v} + \delta_{v_2-v_1} - \delta_{v_2-v_1+v}.
\end{equation}
The value $f(\xi^*) = 5.977657(8)$ was obtained as in Lemma \ref{f_approx_lemma} with 
\begin{equation}
R = \{n_1 v_1 + n_2 v_2 + n_3 v: |n_1|, |n_2| \leq 10, n_3 \in \{0,1\} \}. 
\end{equation}

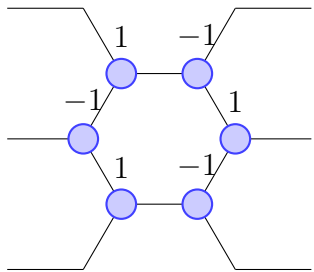
\begin{figure}
  \begin{tikzpicture}[node distance = 1.3cm, auto, place/.style = {circle, 
  thick, draw=blue!75, fill=blue!20}]
  \draw (0, 0)--(1, 0);
  \draw (0, 0)--(-.5, .86602540378);
  \draw (-.5, .86602540378)--(0, 1.73205080757);
  \draw (0, 1.73205080757)--(1, 1.73205080757);
  \draw (1, 1.73205080757)--(1.5, .86602540378);
  \draw (1.5, .86602540378)--(1,0);
  \draw (0,0)--(-.5, -.86602540378);
  \draw (1,0)--(1.5, -.86602540378);
  \draw (-1.5, .86602540378)--(-.5, .86602540378);
  \draw (1.5, .86602540378)--(2.5, .86602540378);
  \draw (0, 1.73205080757)--(-.5, 3*.86602540378);
  \draw (1, 1.73205080757)--(1.5, 3*.86602540378);
  \draw (1.5, 3*.86602540378) --(2.5, 3*.86602540378);
  \draw (-.5, 3*.86602540378)--(-1.5, 3*.86602540378);
  \draw (-.5, -.86602540378)--(-1.5, -.86602540378);
  \draw (1.5, -.86602540378)--(2.5, -.86602540378);
   \node[place] at (0,0)  (lower)[label=$1$]  {};
   \node[place] at (1,0)  (lower)[label=$-1$]  {};
   \node[place] at (-.5, .86602540378)  (lower)[label=$-1$]  {};
   \node[place] at (1.5, .86602540378)  (lower)[label=$1$]  {};
   \node[place] at (0, 1.73205080757)  (lower)[label=$1$]  {};
   \node[place] at (1, 1.73205080757)  (lower)[label=$-1$]  {};
 \end{tikzpicture}
\caption{The extremal configuration for the honeycomb tiling.}\label{fig:honeycomb_extremal}
 \end{figure}

 The following lemma  bounds the number of feasible connected components in the optimizing prevector.
 
 \begin{lemma}\label{height_2_lemma}
  The following optimization programs have the assigned values.  If $|\nu_0| = 1$ then $P(\{0\}, |\nu|) \geq 1.35$.  If $|\nu_0| = 2$ then $P(\{0\}, |\nu|) = 3.5$. If $|\nu_0| = |\nu_v| = 2$ then $P(\{0, v\}, |\nu|) = 4$.  If $|\nu_0| = |\nu_{v_1}| = 2$ then $P(\{0, v_1\}, |\nu|) \geq 5.98$.
 \end{lemma}
\begin{proof}
 The first and last values were calculated in SciPy. 
 
 For the remaining optimization problems, the optimum occurs with all variables non-negative, and hence a boundary value occurs only if a variable is equal to $\frac{1}{2}$.   If $\nu_0 = 2$ the largest value at the optimum is $x_0$, since otherwise the constraint is improved by exchanging $x_0$ and $x_w$ for some $d(w,0)=1$.  This implies that for $w$ such that $d(w,0)=1$, $x_w< \frac{1}{2}$.  By Lagrange multipliers, $\sin(2\pi x_w)$ is constant, so that these values are symmetric about $\frac{1}{4}$.  By averaging $x_0 \geq \frac{1}{3}$ so that $1-c(x_0) \geq 1.5$, and hence there are no pairs of $x_w$ symmetric about $\frac{1}{4}$ and all of the values are at most $\frac{1}{4}$, and hence all equal.  The remaining one variable calculus problem has optimum 3.5.  
 
 In the case where $|\nu_0| = |\nu_v| = 2$, the maximum of $x_v$ is $\frac{1}{2}$. Let $w_1, w_2$ be the two neighbors of 0 other than $v$. Solving the relaxed optimization problem, in which $3x_0 + x_{w_1} + x_{w_2} \geq \frac{3}{2}$,   $x_{w_1} = x_{w_2}$ or $x_{w_1} = \frac{1}{2} - x_{w_2}$, or one of $x_{w_1}$ or $x_{w_2} = \frac{1}{2}$, by Lagrange multipliers.  The values of $1 - c(x_{w_1}) + 1-c(x_{w_2})$ are at least 2 in the latter two cases.  Otherwise, $2x_{w_1} + 3 x_0 \geq 1.5$, and $2(1-c(x_{w_1})) + 1-c(x_0) \geq 2$ by solving the corresponding 1-variable calculus problem.  By symmetry, at least the same value is achieved on the remaining three nodes, so that the optimization has value at least 4.  This is achieved by $x_0 = x_v = \frac{1}{2}$.
 
 The last value was checked in SciPy.

\end{proof}
Several deductions can be made from Lemma \ref{height_2_lemma}.  First, if more than one node of height 2 appears in $\nu$ then they appear adjacent to each other.  This is because, if the nodes appeared at distance at least 3 from each other, then a translation of $P(\{0\}, 2) = 3.5$ could be applied at each node, and the sum would be 7, which is too large.  If the 2's appeared at distance 2 from each other, then a translation and rotation of $P(\{0,v_1\}, 2)\geq 5.98$ could be applied, and this again rules out the configuration.  It follows that at most two 2's can appear, and if two 2's do appear, they appear adjacent to each other. 

Similarly, since $2 P(\{0\}, 1) + P(\{0\}, 2) \geq 2 \times 1.35 + 3.5 = 6.2$, it follows that if there is a connected component containing a 2, there is at most one other connected component. Also, from the first estimate of the lemma, the optimal configuration has at most four connected components, since $5 \times P(\{0\}, 1) > 5 \times 1.35 > 6$.

The next phase of the search estimates $P(S, 1)$ for connected sets $S$.  The following lemma speeds up the computer search by reducing the number of variables which need to be considered having $|x| \geq \frac{1}{4}$.

\begin{lemma}
 Let $S$ be a connected component, and let $w$ satisfy $d(w,S) = 1$ and be such that $w$ has a single neighbor $v \in S$ such that $d(v,w) = 1$. Then the optimizing solution to $P(S,1)$ has $|x_w| \leq \frac{1}{4}$.
\end{lemma}

\begin{proof}
 In the optimizing configuration $x \geq 0$.  Also, $x_w \leq x_v$, since otherwise the values of $x_w$ and $x_v$ can be exchanged which improves the constraint at $v$ and any other constraints containing $x_v$.  If $x_w > \frac{1}{4}$ then $3x_v + \sum_{d(u,v)=1} x_u \geq x_w + 3 x_v > 1$, and all of the constraints will still be satisfied if $x_w$ is reduced.  Hence $x_w \leq \frac{1}{4}$.
\end{proof}

 Since $P(S,1) \leq P(S',1)$ when $S \subset S'$, all connected components $S$ with $P(S,1) < \gamma_{\hex}$ were enumerated by building the components one vertex at a time by adding a vertex at distance at most 2 from the existing configuration.  All possible such components were enumerated in SciPy, and any such configuration has at most 8 vertices.  Next, all assignments of $1$ and $-1$ to the vertices of a connected component which cause the configuration to be in $C^2$ were tested, and for each such configuration, the value of $f(\xi)$ was estimated by calculating the first few values of $\xi$ near 0.  By doing so, it was verified that the configuration $\nu_0$ is the minimizing configuration with a single connected component, and height bounded by 1.  To prove that this is the overall minimizer, it remains to rule out nodes of height 2, and several connected components.
 \subsubsection*{Case of several connected components of height 1}
 The following Lemma reduces the number of connected components which need to be considered.
 \begin{lemma}
  The following optimization problems have the claimed values, $P(\{0, v\}, 1)  \geq 1.87$, $P(\{0, v_1\}, 1) \geq 2.59$. For $\nu_0 = 1$, $\nu_v = 0$, $Q_2(\{0,v\}, \nu) \geq 1.72$.
    Let $S_1 = \{w: d(w,0) \leq 1\}$ and $S_2 = \{w:  d(w,0) \leq 2\}$.  Then $Q(S_1, \delta_0) \geq 2.92$ and $Q(S_2, \delta_0) \geq 4.56$.
  
  \end{lemma}
\begin{proof}
 These values were checked in SciPy.
\end{proof}

\subsubsection*{Case of four connected components}
The case of four connected components is ruled out as follows.  No connected component contains a pair of nodes at distance two from each other, since $P(\{0, v_1\}, 1) + 3 P(\{0\}, 1) \geq 2.59+ 3 \times 1.35  > 6$.  A connected component containing adjacent nodes has two non-zero nodes.  Since the configuration must be in $C^2(\sT)$, the number of nodes is even, and hence if this occurs, at least two components are of this type. Since $2P(\{0, v\}, 1) + 2 P(\{0\}, 1) \geq 2 \times 1.87 + 2 \times 1.35 >6 $ this, also, does not achieve the minimum. This reduces to the case of four singleton components of height 1.  Notice that each of the singletons has distance at least 3 from all of the others.  

The following lemma is used to show that no two singletons can have pairwise distance at least 5.
\begin{lemma}\label{separating_line_lemma}
 Any half plane through 0 contains either $\{0, v, v_1, v_2\}$ or one of its rotations by 120 degrees.
\end{lemma}
\begin{proof}
 The extremal line passes through the endpoint of one branch of the tree extending from 0, and eliminates two, but not all three, of the branches, see Figure \ref{fig:branch_length_2}.
 
 \begin{figure}
   \begin{tikzpicture}[node distance = 1.3cm, auto, place/.style = {circle, 
   thick, draw=blue!75, fill=blue!20}]
  \draw (0, 0)--(1, 0);
  \draw (1,0)--(1.5, .86602540378);
  \draw (1,0)--(1.5, -.86602540378);
  \draw (0,0)--(-.5, .86602540378);
  \draw (-.5, .86602540378)--(-1.5, .86602540378);
  \draw (-.5, .86602540378)--(0, 2*.86602540378);
  \draw (0,0)--(-.5, -.86602540378);
  \draw (-.5, -.86602540378)--(-1.5, -.86602540378);
  \draw (-.5, -.86602540378)--(0, -2*.86602540378);
  \node[place] at (0,0)  (lower)[label=$0$]  {};
  \draw[dashed] (0, 2)--(0,-2);
  \end{tikzpicture}
\caption{In the honeycomb tiling, a line through a node can cut off two but not three branches of length 2.}\label{fig:branch_length_2}
 \end{figure}
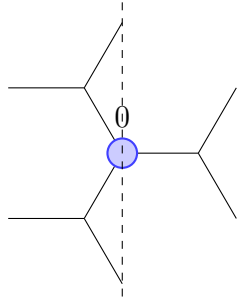
\end{proof}

\begin{lemma}\label{two_singletons_dist_5}
 The extremal configuration does not have two connected components which are singletons separated by a graph distance greater than 4.
\end{lemma}
\begin{proof}
Suppose that two such singletons exist, say at  $z_1$ and $z_2$.  Since the convex hulls of the distance 1 and 3 neighborhoods of a node in the honeycomb tiling do not contain any further nodes of the tiling (see Figure \ref{fig:dist_3_convex_hull}), the distance 3 neighborhood of $z_1$ and the distance one neighborhood of $z_2$ are separated by a line.  

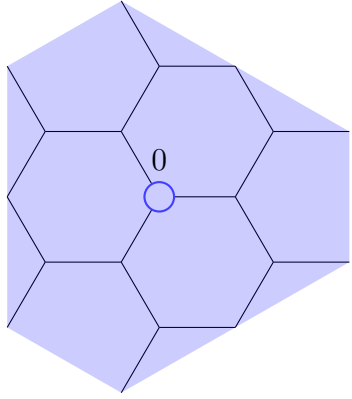
\begin{figure}
   \begin{tikzpicture}[node distance = 1.3cm, auto, place/.style = {circle, 
   thick, draw=blue!75, fill=blue!20}]
  \draw (0, 0)--(1, 0);
  \draw (1,0)--(1.5, .86602540378);
  \draw (1,0)--(1.5, -.86602540378);
  \draw (1.5, .86602540378) -- (1, 2*.86602540378);
  \draw (1.5, .86602540378)-- (2.5, .86602540378);

  \draw (1.5, -.86602540378)--(1, -2*.86602540378);
  \draw (1.5, -.86602540378)--(2.5, -.86602540378);
  
  \draw (0,0)--(-.5, .86602540378);
  \draw (-.5, .86602540378)--(-1.5, .86602540378);
  \draw (-1.5, .86602540378)--(-2, 2*.86602540378);
  \draw (-1.5, .86602540378)--(-2, 0);

  \draw (-.5, .86602540378)--(0, 2*.86602540378);
  \draw (0, 2*.86602540378)--(1, 2*.86602540378);
  \draw (0, 2*.86602540378)--(-.5, 3*.86602540378);

  \draw (0,0)--(-.5, -.86602540378);
  \draw (-.5, -.86602540378)--(-1.5, -.86602540378);
  \draw (-1.5,-.86602540378)--(-2, 0);
  \draw (-1.5, -.86602540378)--(-2, -2*.86602540378);

  \draw (-.5, -.86602540378)--(0, -2*.86602540378);
  \draw (0, -2*.86602540378)--(-.5, -3*.86602540378);
  \draw (0, -2*.86602540378)--(1, -2*.86602540378);
  
  \fill[fill= blue, opacity = 0.2] (-.5, -3*.86602540378) -- (2.5, -.86602540378) -- (2.5, .86602540378) -- (-.5, 3*.86602540378) -- (-2, 2*.86602540378)-- (-2, -2*.86602540378) ;
  
  \node[place] at (0,0)  (lower)[label=$0$]  {};
  \end{tikzpicture}
\caption{The convex hull of the distance 3 neighborhood of a point in the honeycomb tiling contains only vertices in the neighborhood.}\label{fig:dist_3_convex_hull}
 \end{figure}
 
By Lemma \ref{separating_line_lemma}, in fact, the distance one neighborhood of $z_2$ and one of its neighbors does not intersect the distance 3 neighborhood of $z_1$.  Since neither remaining node in the configuration intersects the distance 2 neighborhood of either $z_1$ or $z_2$, it is possible to apply a translation and rotation of $Q(S_2, \delta_0)$ at $z_1$ and, for $\nu_0 = 1$, $\nu_v = 0$, a translation and rotation of $Q(\{0, v\}, \nu)$ at with 0 translated to $z_2$, to obtain a value  at least $Q(S_2, \delta_0) + Q(\{0,v\}, \nu) \geq 4.56 + 1.72>6 $. 
\end{proof}

Thus all of the points have mutual distance at most 4 and at least 3.  
To be in $C^2(\sT)$, two of the points are positive and two negative, and the two pairs have the same center of mass.  Up to translation and symmetry, there are two types of pairs of points at distance 3, and two types of pairs of points at distance 4, see Figure \ref{fig:dist_3_4_points}.

\begin{figure}
  \begin{tikzpicture}[node distance = 1.3cm, auto, place/.style = {circle, 
  thick, draw=blue!75, fill=blue!20}]
  \draw (0, 0)--(1, 0);
  \draw (0, 0)--(-.5, .86602540378);
  \draw (-.5, .86602540378)--(0, 1.73205080757);
  \draw (0, 1.73205080757)--(1, 1.73205080757);
  \draw (1, 1.73205080757)--(1.5, .86602540378);
  \draw (1.5, .86602540378)--(1,0);
  \draw (0,0)--(-.5, -.86602540378);
  \draw (1,0)--(1.5, -.86602540378);
  \draw (-1.5, .86602540378)--(-.5, .86602540378);
  \draw (1.5, .86602540378)--(2.5, .86602540378);
  \draw (0, 1.73205080757)--(-.5, 3*.86602540378);
  \draw (1, 1.73205080757)--(1.5, 3*.86602540378);
  \draw (1.5, 3*.86602540378) --(2.5, 3*.86602540378);
  \draw (-.5, 3*.86602540378)--(-1.5, 3*.86602540378);
  \draw (-.5, -.86602540378)--(-1.5, -.86602540378);
  \draw (1.5, -.86602540378)--(2.5, -.86602540378);
  \draw (2.5, .86602540378)--(3, 2*.86602540378);
  \draw (2.5, 3*.86602540378)--(3, 2*.86602540378);
  \draw (2.5, .86602540378)--(3, 0);
  \draw (2.5, -.86602540378)--(3,0);
   \node[place] at (0,0)  (lower)[label=$0$]  {};
   \node[place] at (1, 1.73205080757)  (lower)[label=$A_3$]  {};
   \node[place] at (2.5, .86602540378)  (lower)[label=$A_3'$]  {};
   \node[place] at (1.5, 3*.86602540378)  (lower)[label=$A_4$]  {};
   \node[place] at (3, 2*.86602540378)  (lower)[label=$A_4'$]  {};
 \end{tikzpicture}
\caption{Up to rotation and translation, there are two pairs of nodes at distance 3 and two pairs of nodes at distance 4 in the honeycomb tiling.  Modulo the tiling, the midpoints of the pairings are inequivalent.}\label{fig:dist_3_4_points}
 \end{figure}
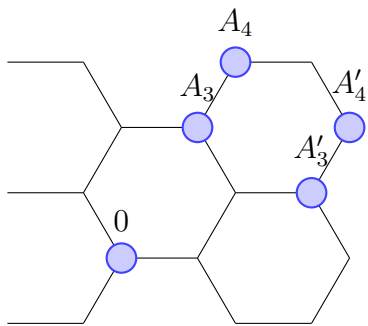
 
 Each of the types of pairs has an inequivalent type of center of mass, so that both positive and negative pair must be of the same type.  This cannot be obtained while keeping all nodes at distance at least 3 from each other.

\subsubsection*{Case of 3 connected components}
 Next consider configurations with three connected components.  There cannot be two connected components containing nodes at distance 2, since $2P(\{0,v_1\},1) + P(\{0\}, 1) \geq 2 \times 2.59 + 1.35 > 6$.  Also, if there is a connected component with nodes at distance 2, there is at most one further connected component with at least two nodes, since $P(\{0, v_1\},1) + 2P(\{0,v\},1) \geq 2.59 + 2 \times 1.87> 6$.  Hence there is at least one singleton.  First suppose that there is one singleton, one adjacent pair, and one further component which has a pair of nodes at distance 2.  The last component must have odd size, hence has size at least 3.  By checking the $P$ values, the only possibility is that the component is, up to symmetry, given by $\{0, v, v_1\}$.  
 \begin{lemma}
  The following signed optimization has value $Q(\{0,v\}, 1) \geq 3.79$. If $\nu_0 = \nu_{v_1} = 1$ and $\nu_v = -1$, $Q(\{0,v,v_1\}, \nu) \geq 2.59$.
 \end{lemma}
Using this lemma, it follows that the sign of the node at position $v$ is opposite of the sign of the nodes at 0 and $v_1$, since $Q(\{0,v\},1) + P(\{0,v\},1) + P(\{0\},1) \geq 3.79 + 1.87 + 1.35>6$. Say $\nu_{v}=-1$. Also, the pair of adjacent nodes have opposite sign, since $Q(\{0,v\}, 1) + Q(\{0,v,v_1\},\nu) + P(\{0\},1) \geq 3.79 + 2.59 + 1.35 > 6$. Fix the size 3 configuration at $\{0, v, v_1\}$ with $\nu_0 = \nu_{v_1} = 1$, $\nu_v = -1$.  This configuration has signed sum $v_1 - v$, and the singleton has sign $-1$.  Since the pair has sum which is a vector of length $\|v\|$, the possible locations for the singleton are displayed for the net sum to be 0.
For the configuration to have mean 0, this places the singleton on some corner of the hexagon with center $v_1-v$, but no such node has distance at least 3 from the configuration $\{0, v, v_1\}$, see Figure \ref{fig:0_v_v1_configuration}.

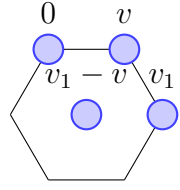
\begin{figure}
  \begin{tikzpicture}[node distance = 1.3cm, auto, place/.style = {circle, 
  thick, draw=blue!75, fill=blue!20}]
  \draw (0, 0)--(1, 0);
  \draw (1,0)--(1.5, -.86602540378);
  \draw (0,0)--(-.5, -.86602540378);
  \draw (-.5, -.86602540378)--(0, -2*.86602540378);
  \draw (0, -2*.86602540378)--(1, -2*.86602540378);
  \draw (1, -2*.86602540378)--(1.5, -.86602540378);
  
  \node [place] at (.5, -.86602540378) (lower)[label = $v_1-v$] {};
  \node [place] at (1,0) (lower) [label = $v$] {};
  \node [place] at (1.5, -.86602540378) (lower) [label = $v_1$] {};
   \node[place] at (0,0)  (lower)[label=$0$]  {};
 \end{tikzpicture}
\caption{The alternating sign configuration on $\{0,v,v_1\}$ has moment $v_1-v$.}\label{fig:0_v_v1_configuration}
 \end{figure}
 
This reduces to the case of a pair of singletons together with a configuration having at least two nodes.  There are no configurations $S$ of size at least 6 satisfying $P(S,1) + 2\times 1.35  < 6$ and three such configurations of size 4, which are pictured in Figure \ref{fig:3_size_4_components}.
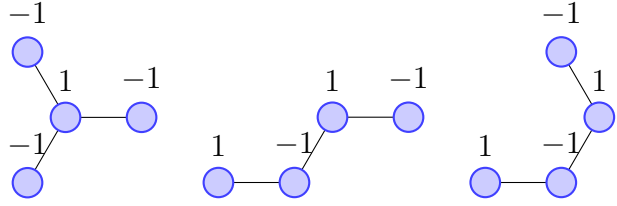
\begin{figure}
 \begin{tikzpicture}[node distance = 1.3cm, auto, place/.style = {circle, 
 thick, draw=blue!75, fill=blue!20}]
  \draw (0, 0)--(1, 0);
  \draw (0, 0)--(-.5, .86602540378);
  \draw (0,0)--(-.5, -.86602540378);
   \node[place] at (0,0)   (lower)[label = $1$] {};
   \node[place] at (1, 0) (lower) [label = $-1$]  {};
   \node[place] at (-.5, .86602540378) (lower) [label = $-1$]  {};
   \node[place] at (-.5, -.86602540378) (lower) [label = $-1$]  {};
 \end{tikzpicture}\;\;
   \begin{tikzpicture}[node distance = 1.3cm, auto, place/.style = {circle, 
   thick, draw=blue!75, fill=blue!20}]
  \draw (0, 0)--(1, 0);
  \draw (1,0)--(1.5, .86602540378);
  \draw (1.5, .86602540378)--(2.5, .86602540378);
   \node[place] at (0,0) (lower) [label = $1$]   {};
   \node[place] at (1, 0) (lower) [label = $-1$] {};
   \node[place] at (1.5, .86602540378) (lower) [label = $1$]  {};
   \node[place] at (2.5, .86602540378) (lower) [label = $-1$]  {};
 \end{tikzpicture}\;\;
    \begin{tikzpicture}[node distance = 1.3cm, auto, place/.style = {circle, 
    thick, draw=blue!75, fill=blue!20}]
  \draw (0, 0)--(1, 0);
  \draw (1,0)--(1.5, .86602540378);
  \draw (1.5, .86602540378)--(1, 2*.86602540378);
   \node[place] at (0,0)  (lower) [label = $1$]  {};
   \node[place] at (1, 0) (lower) [label = $-1$]  {};
   \node[place] at (1.5, .86602540378) (lower) [label = $1$] {};
   \node[place] at (1, 2*.86602540378) (lower) [label = $-1$]  {};
 \end{tikzpicture}
\caption{There are three size four components which can be paired with a pair of singletons based upon their $P$ values. Adjacent nodes have opposite signs.}\label{fig:3_size_4_components}
 \end{figure}
 Since $Q(\{0,v\},1) \geq 3.79$ it is not possible that adjacent nodes have the same sign.  Thus in the latter two pictures there are an equal number of positive and negative nodes, and for two singletons to be added that make a configuration in $C^2(\sT)$, they have opposite sign.  In the case of the middle picture of Figure \ref{fig:3_size_4_components}, the two opposite signed nodes differ by $2v = v_1 + v_2 - v$.
 \begin{lemma}\label{diameter_hex_lemma}
  Let $\nu_v = -1, \nu_{v_1+v_2} = 1$, $\nu_w = 0$ otherwise, and let $S = \{w: d(w, \{v, v_1+v_2\}) \leq 1\}$.  Then $Q(S, \nu) \geq 5.03$.
 \end{lemma}
The middle picture may now be ruled out as follows. Fix an orientation  by setting $\nu_0 = \nu_{v_2} = 1$ and $\nu_{v} = \nu_{v_2 +v} = -1$.  Thus the pair of singletons differ by $2v = v_1 + v_2 - v$, which is possible only if they form horizontal nodes which are on the diameter of a hexagon. Applying the optimization of Lemma \ref{diameter_hex_lemma} with $v$ and $v_1+ v_2$ translated to correspond with the two highlighted nodes uses variables at a configuration pictured.

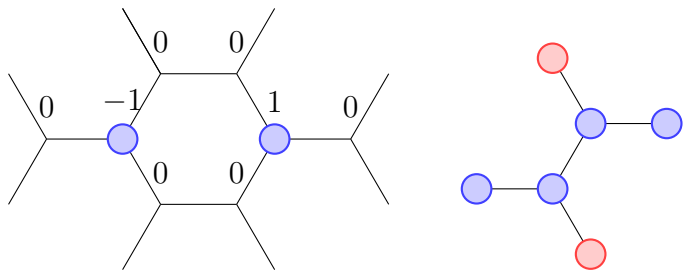
\begin{figure}
    \begin{tikzpicture}[node distance = 1.3cm, auto, place/.style = {circle, 
    thick, draw=blue!75, fill=blue!20}]
  \draw (0, 0)--(1, 0);
  \draw (0,0)--(-.5, .86602540378);
  \draw (0,0)--(-.5, -.86602540378);
  \draw (1,0)--(1.5, .86602540378);
  \draw (1,0)--(1.5, -.86602540378);
  \draw (1.5, .86602540378)--(1, 2*.86602540378);
  \draw (1.5, -.86602540378)--(1, -2*.86602540378);
  \draw (1.5, .86602540378)--(1, 2*.86602540378);
  \draw (1.5, .86602540378)--(2.5, .86602540378);
  \draw (1.5, -.86602540378)--(2.5, -.86602540378);
  \draw (2.5, .86602540378)--(3, 0);
  \draw (2.5, -.86602540378)--(3,0);
  \draw (3,0)--(4,0);
  \draw (4,0)--(4.5, .86602540378);
  \draw (4,0)--(4.5, -.86602540378);
  \draw (2.5, .86602540378)--(3, 2*.86602540378);
  \draw (2.5, -.86602540378)--(3, -2*.86602540378);
   \node[place] at (1,0)  (lower)  [label = $-1$]{};
   \node[place] at (3,0) (lower) [label = $1$]{};
   \node at (0,0) (lower) [label = $0$]{};
   \node at (4,0) (lower) [label = $0$]{};
   \node at (1.5, .86602540378) (lower) [label = $0$] {};
   \node at (2.5, .86602540378) (lower) [label = $0$] {};
   \node at (1.5, -.86602540378) (lower) [label = $0$] {};
   \node at (2.5, -.86602540378) (lower) [label = $0$] {};

 \end{tikzpicture}\;\;\;\;\;\;\;
   \begin{tikzpicture}[node distance = 1.3cm, auto, place/.style = {circle, 
   thick, draw=blue!75, fill=blue!20}]
  \draw (0, 0)--(1, 0);
  \draw (1,0)--(1.5, .86602540378);
  \draw (1.5, .86602540378)--(2.5, .86602540378);
  \draw (1,0)--(1.5, -.86602540378);
  \draw (1.5, .86602540378) -- (1, 2*.86602540378);
   \node[place] at (0,0)    {};
   \node[place] at (1, 0)  {};
   \node[place] at (1.5, .86602540378)  {};
   \node[place] at (2.5, .86602540378)   {};
   \node[place, draw = red!75, fill = red!20] at (1.5, -.86602540378) {};
   \node[place, draw = red!75, fill = red!20] at (1, 2*.86602540378) {};
 \end{tikzpicture}
\caption{The picture on the left shows those variables used in the optimization program of Lemma \ref{diameter_hex_lemma}.  If both red nodes in the picture on the right are used in this optimization, then one of the blue nodes is used as well.}\label{fig:red_blue_figure}
 \end{figure}

Note that this configuration contains all of the nodes in its convex hull.  In particular, in Figure \ref{fig:red_blue_figure} one of the two red nodes is not used in the optimization of Lemma \ref{diameter_hex_lemma}, since if both red nodes were used, one of the blue nodes in the figure would be used, also, which is impossible since the singletons have distance at least 3 from the blue nodes.

 It follows that $P(\{0\},1)$ may be applied using variables at three blue nodes and one red node,  together with the optimization of Lemma \ref{diameter_hex_lemma}, which gives a value at least $1.35 + 5.03 > 6$.  
 
 To rule out the final picture of Figure \ref{fig:3_size_4_components}, fix the picture by setting $\nu_0 = \nu_{v_2} = 1$, $\nu_v = \nu_{-v_1 + v_2 + v} =-1$.
  Thus the signed sum of the picture is $v_1-2v = v-v_2$, which is a segment of graph length one in the tiling.  Thus it is impossible to place the opposite signed singletons at distance at least 3 from each other and obtain a configuration of mean 0.

 In the first picture of Figure \ref{fig:3_size_4_components}, fix the picture by making center node at 0.  The two singletons $x_1, x_2$ have the same sign as the center node, and have it as their center of mass. This is only possible if each has distance at least 4 from the central node. Also, for $x_1 = -x_2$, it follows that both $x_1$ and $x_2$ are in the triangular lattice, say at $\pm(n_1v_1 + n_2 v_2)$.  The graph distance in the hex tiling between points in the triangular lattice is twice the graph distance in the triangular lattice, since two moves in the hex lattice is one move in the triangular lattice.  It follows that the distance between $x_1$ and $x_2$ is at least 6, since the distance 2 neighborhood of $x_1$ in the triangular lattice is convex (a hexagon).    This case is now ruled out by applying  Lemma \ref{two_singletons_dist_5} which rules out a pair of singletons at hex graph distance at least 5 from each other.

 It remains to consider the case in which there are two singletons and a component of two nodes.  Necessarily, the two nodes have the same sign, since otherwise the two singletons would have the opposite sign and have the same distance to remain in $C^2(\sT)$.  This eliminates the case that the two nodes are adjacent, since $Q(\{0,v\}, 1) + 2P(\{0\},1) \geq 3.79 + 2 \times 1.35 > 6$, so assume that they are at 0 and $v_1$.  The two singletons thus have center of mass $\frac{v_1}{2}$.  
 The two singletons must lie in the triangular lattice generated by $v_1, v_2$, since a component of $\frac{v}{2}$ or $v$ has $v_1$ or $v_2$ coordinates which have denominator divisible by 3.  
 Since each singleton has hex graph distance at least 3 from $\{0, v_1\}$, it follows that their distance from $\{0, v_1\}$ in the triangular lattice is at least 2. Let the two singletons be at $x$, and $v_1-x$.  Let $B_r(x)$ be the ball of radius $r$ in the triangular lattice distance centered at $x$, which is a hexagon.  If this ball contains $v_1-x$, then its convex hull contains $\frac{v_1}{2}$ and hence both 0 and $v_1$, so that $r \geq 3$, and hence $x$ and $v_1-x$ have hex distance at least 6.  The configuration is now ruled out by applying  Lemma \ref{two_singletons_dist_5} which forbids singletons at hex graph distance at least 5 from each other. 
 
 There is one further case to be eliminated with three connected components, in which no component has diameter more than 1.  The case of a pair of adjacent vertices with two singletons was already ruled out.  This leaves only the case of 3 pairs of adjacent vertices.  Note that each of these would necessarily consist of pairs having opposite sign, since a pair of the same sign contribute $Q(\{0,v\},1) \geq 3.79$ and $2 \times P(\{0,v\},1) + Q(\{0,v\},1)>6$.  The case is limited further by the following lemma.
 
   \begin{lemma}\label{one_node_expansion_boundary}
  The following optimization programs have the stated values. For $S_1 = \{w : d(w, \{0,v\}) \leq 1$ and for $\nu_0 = 1$, $\nu_v = -1$, $\nu_w = 0$ otherwise, $Q(S_1, \nu) \geq 2.99$.  For $S_2 = \{w : d(w, \{0, v_1\}) \leq 1\}$ and $\nu_0 = 1$, $\nu_{v_1} = -1$, $\nu_w = 0$ otherwise, $Q(S_2, \nu) \geq 4.47$.  If $\nu_0 = 1$ and $\nu_{v_1} = 1$, $\nu_w = 0$ otherwise, then $Q(S_2, \nu) > 7.2$.
 \end{lemma}

 \begin{proof}
  These values were verified in SciPy.
 \end{proof}
 
 It follows that none of the pairs of adjacent vertices has distance at least 4 from the others, since otherwise the optimization involving $S_1$ could be applied at that pair and $P(\{0,v\},1)$ at the other two.  All ways of arranging three pairs which are distance 2 completely disconnected, but distance 3 connected were enumerated, and all assignments of signs producing a configuration in $C^2(\sT)$ were checked, none attains the optimum.

 \subsubsection*{Case of 2 connected components}
 
 It remains to consider the case of two connected components, and components with a node of height 2. First the case of two connected components with no nodes of height 2 is considered, and these are categorized by the size of the largest component.  
 
 By exhaustively checking with $P(S, 1)$, there does not exist a component of size at least 7 which can be paired with a second component.  A component of size 6 must be paired with a component of even size.  There are no components $S$ of size 6 for which $P(S, 1) + P(\{0, v_1\}, 1) < 6$ and hence a component of size 6 may only be paired with a pair of adjacent vertices.  The only configuration of size 6 with $P(S, 1) + P(\{0,v\},1) <6$ has shape given in Figure \ref{fig:hexagon}.
 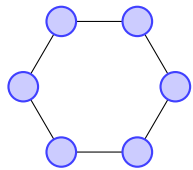
\begin{figure}
   \begin{tikzpicture}[node distance = 1.3cm, auto, place/.style = {circle, 
   thick, draw=blue!75, fill=blue!20}]
  \draw (0, 0)--(1, 0);
  \draw (0, 0)--(-.5, .86602540378);
  \draw (-.5, .86602540378)--(0, 1.73205080757);
  \draw (0, 1.73205080757)--(1, 1.73205080757);
  \draw (1, 1.73205080757)--(1.5, .86602540378);
  \draw (1.5, .86602540378)--(1,0);
   \node[place] at (0,0)    {};
   \node[place] at (1,0)    {};
   \node[place] at (-.5, .86602540378)    {};
   \node[place] at (1.5, .86602540378)   {};
   \node[place] at (0, 1.73205080757)    {};
   \node[place] at (1, 1.73205080757)    {};
 \end{tikzpicture}
 \caption{The only configuration on six nodes which can be paired with a pair of adjacent nodes.}
 \label{fig:hexagon}\end{figure}

 If the pair of points have the same sign, then the number of points on the hexagon with positive and negative signs is unequal, and thus there are adjacent points on the hexagon with the same sign.  Thus, applying $Q(\{0,v\}, 1) \geq 3.79$ twice shows that this configuration is not the minimum.  Hence the pair of vertices have opposite signs and so there are equal positive and negative signs on the hexagon.  The configuration in which positive and negative signs alternate on the hexagon is in $C^2(\sT)$ itself, so cannot be paired with the pair and remain in $C^2(\sT)$.  Any other configuration is found by flipping an equal number of positive and negative signs and hence has a signed sum which, if non-zero, is in the lattice generated by $\{2v, 2v_1, 2v_2\}$, and hence still is not equal to the sum of the pair of adjacent vertices with opposite sign.
 
 There are no pairs of components $S_1, S_2$ with $S_1$ of size 3 and $S_2$ of size 5 such that $P(S_1, 1) + P(S_2, 1) < 6$. Note that an assignment of signs to the nodes in a size 5 configuration determines the location of a singleton which may be added to form a configuration in $C^2(\sT)$.  All such configurations were tested, and none gives the optimum. 
 
 There are no pairs $S_1, S_2$, both of size 4, for which $P(S_1, 1) + P(S_2, 1) < 6$. Lemma  \ref{one_node_expansion_boundary} is used to control pairs in which one component has size 4 and the other has size 2.
 It was verified in SciPy that there is no configuration $S$ of size 4 such that $P(S,1) + 2.99 < 6$.  Combined with Lemma \ref{one_node_expansion_boundary} this proves that a component of size 2 has distance exactly 3 from a component of size 4.  All size four components whose $P$ value can be combined with the $P$ value of either a pair of adjacent nodes, or a pair of nodes at distance 2 from each other were enumerated.  All ways of combining the size 4 and 2 components at distance 3 were enumerated, and all ways of assigning signs to the vertices to obtain a configuration in $C^2(\sT)$ were enumerated.  None of these achieved the optimum.

 Among those configurations on 3 nodes, only the configuration given up to symmetry by $\{0, v, v_1\}$ has $P$ value less than 3, and hence, if two components of size 3 are combined, one has shape $\{0, v, v_1\}$.  Recall that if $\nu_0 = \nu_{v_1} = 1$, $\nu_v = -1$ then $Q(\{0, v, v_1\}, \nu) \geq 2.59$ and $Q(\{0, v\}, 1) \geq 3.79$. Thus neither component can have adjacent nodes of the same sign. Up to symmetry, there are four configurations with which $\{0, v, v_1\}$ can be paired.  These are pictured in Figure \ref{fig:size_3_components}. The first, second and fourth configurations in the figure have adjacent nodes which receive opposite signs, so that the signed sum of the vertices has distance one from the remaining node of the configuration (up to sign).  To be paired with $\{0, v, v_1\}$ this places one node of the configuration on the hexagon that has $\{0, v, v_1\}$ as vertices, but then this node does not have distance 3 from $\{0, v, v_1\}$.  This eliminates all but the third configuration. Let $\nu_0 = 1$, $\nu_{v_1} = \nu_{v_2} = 1$.  It was verified in SciPy that $Q(\{0, v_1, v_2\}, \nu) \geq 3.84$, which rules out the third configuration.
 
  \begin{figure}
   \begin{tikzpicture}[node distance = 1.3cm, auto, place/.style = {circle, 
   thick, draw=blue!75, fill=blue!20}]
  \draw (0, 0)--(1, 0);
  \draw (1,0) -- (1.5, - .86602540378);
  \node [place] at (1,0) (lower) [label = $v$] {};
  \node [place] at (1.5, -.86602540378) (lower) [label = $v_1$] {};
   \node[place] at (0,0)  (lower)[label=$0$]  {};
 \end{tikzpicture} \;\;
    \begin{tikzpicture}[node distance = 1.3cm, auto, place/.style = {circle, 
    thick, draw=blue!75, fill=blue!20}]
  \draw (0, 0)--(1, 0);
  \draw (1,0) -- (1.5, - .86602540378);
  \draw (1.5, -.86602540378)--(2.5, -.86602540378);
  \node [place] at (1,0) (lower) [label = $v$] {};
  \node [place] at (2.5, -.86602540378) (lower) [label = $v_1+v$] {};
   \node[place] at (0,0)  (lower)[label=$0$]  {};
 \end{tikzpicture}
     \begin{tikzpicture}[node distance = 1.3cm, auto, place/.style = {circle, 
     thick, draw=blue!75, fill=blue!20}]
  \draw (0, 0)--(1, 0);
  \draw (1,0) -- (1.5, - .86602540378);
  \draw (1,0) -- (1.5,  .86602540378);
  \node [place] at (1.5, .86602540378) (lower) [label = $v_2$] {};
  \node [place] at (1.5, -.86602540378) (lower) [label = $v_1$] {};
   \node[place] at (0,0)  (lower)[label=$0$]  {};
 \end{tikzpicture}\;\;
      \begin{tikzpicture}[node distance = 1.3cm, auto, place/.style = {circle, 
      thick, draw=blue!75, fill=blue!20}]
  \draw (0, 0)--(1, 0);
  \draw (1,0) -- (1.5, - .86602540378);
  \draw (1,0) -- (1.5,  .86602540378);
  \draw (1.5, .86602540378) -- (2.5, .86602540378);
  \node [place] at (2.5, .86602540378) (lower) [label = $v_2+v$] {};
  \node [place] at (1.5, -.86602540378) (lower) [label = $v_1$] {};
   \node[place] at (1,0)  (lower)[label=$v$]  {};
 \end{tikzpicture}
\caption{Up to symmetry, all components which can appear as the the second size 3 component in a configuration are shown.  Adjacent nodes must have opposing signs.}\label{fig:size_3_components}
 \end{figure}
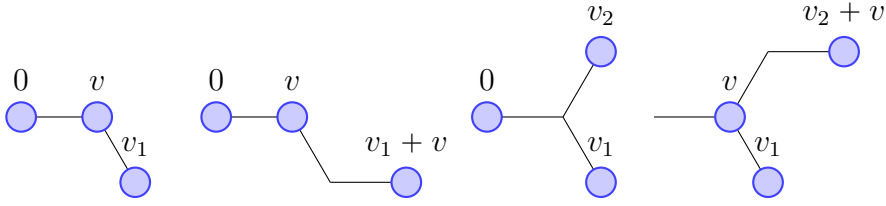
 
If signs are assigned to a configuration of size 3 which is to be paired with a singleton, to reach a configuration in $C^2(\sT)$, two of the nodes in the size 3 configuration have one sign, and the third has the other, so that there are nodes at distance at most 2 with opposite signs.  The singleton would then have the same distance from the third node, and hence be connected to the component of size 3.  This rules out a singleton and a component of size 3.  
 
 This reduces to the case of two configurations both of size 2, which necessarily are both adjacent pairs, or pairs at distance two, and necessarily are pairs with opposing signs in order to be in $C^2(\sT)$.

A consequence of Lemma \ref{one_node_expansion_boundary} is that, if the configuration consists of two pairs of adjacent nodes, then their distance is less than 5, since otherwise after a translation, $Q(S_1, \nu)$ could be applied at each pair.  Assume that one pair has been translated to $0, v$.  Then the other pair has the form $w, w+v$ where $w$ is in the triangular lattice.  The distance between $\{w, w+v\}$ and $\{0, v\}$ is either twice the distance from 0 to $w$ in the triangular lattice, or one less, and hence the distance in the triangular lattice from 0 to $w$ is 2.  There are 12 such choices, and each verified not to obtain the optimum.  It also follows from the lemma that if there are two pairs of nodes which have distance 1 from each other, then the distance between the pairs is less than 4, and hence equal to 3, since otherwise $Q(S_2, \nu)$ could be applied at one pair, and $P(\{0, v_1\}, 1)$ at the other.  The pairs were enumerated, and none obtains the optimum.

\subsubsection*{Case of a single node of height 2.}  \begin{lemma}
  The optimization problem has value $Q(\{0, v\}, 2) \geq 12$.  For $\nu_0= 2$, $\nu_v = 1$, $Q(\{0, v\}, \nu) \geq 7.8$.
 \end{lemma}
\begin{proof}
 This was checked in SciPy.
\end{proof}
It follows that if there is a node of height 2, any node appearing adjacent to it appears with the opposite sign.

Recall that $P(\{0\}, 2) + 2 P(\{0\},1) > 3.5 + 2 \times 1.35 > 6$, and hence if there is a component with a node of height 2, there are at most 2 connected components.  A list of all connected components having a single node of height 2, and $P$ value at most 6 was constructed. The largest components in this list had size 7.  First consider the case of two connected components, one of which has a node of height 2. By examining the list of connected components, all of whose nodes have height 1, the minimum $P$ value of a component with at least 3 nodes is at least $2.59$.  Since $P(\{0\}, 2) = 3.5$, it follows that the component with heights 1 is either a singleton or a pair of nodes which are adjacent, or at distance 1.  The $P$ values of these configurations are at least 1.35, 1.87 and 2.59. The case of $2.59$ may be ruled out by comparing with $\gamma_{\hex}$.  By first reducing the configurations to those with $P$ value at most $6-1.35 = 4.65$ it follows that the  component containing a node of height 2 has size at most 5. Furthermore, the component of size 5 must be paired with a component of even size, and no component of size 5 has $P$ value that meets this requirement.  The case that the component containing a 2 has size 1 or 2 can be ruled out, since the resulting configuration cannot be made in $C^2(\sT)$.  This reduces to the cases of a component of size 3 or 4. The only configuration of size 3 which may be paired with a pair of adjacent nodes has shape 

\begin{figure}
   \begin{tikzpicture}[node distance = 1.3cm, auto, place/.style = {circle, 
   thick, draw=blue!75, fill=blue!20}]
  \draw (0, 1.73205080757)--(-.5, .86602540378);
  \draw (0, 1.73205080757)--(1, 1.73205080757);
   \node[place] at (-.5, .86602540378)  (lower)[label=$-1$]  {};
   \node[place] at (0, 1.73205080757) (lower)[label = $2$]   {};
   \node[place] at (1, 1.73205080757) (lower)[label = $-1$]   {};
 \end{tikzpicture}
 \begin{tikzpicture}[node distance = 1.3cm, auto, place/.style = {circle, 
 thick, draw=blue!75, fill=blue!20}]
  \draw(0,0)--(-.5, .86602540378);
  \node[place] at (-.5, .86602540378) (lower) [label=$-1$] {};
  \node[place] at (0,0) (lower) [label=$1$] {};
 \end{tikzpicture}
\caption{A size three component which may be paired with a pair of adjacent vertices has the configuration shown.}\label{fig:size_3_2_component}
 \end{figure}
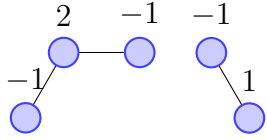
 
For $\nu_0 = 2$, $\nu_{v} = \nu_{-v_1+v}=1$, $P(\{0, v, v-v_1\} , \nu) \geq P(\{0\}, 2) \geq 3.5$.   
To be in $C^2(\sT)$, this configuration must be paired with a configuration which is a translate of the one pictured in Figure \ref{fig:size_3_2_component}.  By applying Lemma \ref{one_node_expansion_boundary} to the pair of adjacent nodes, it follows that this pair has distance exactly 3 from the first configuration, since if the pair was at a greater distance than 3,  $Q(S_1, \nu) \geq 2.99$ could be applied at the pair.  By subtracting the graph Laplacian at 0 from the height 2 configuration, obtain a second pair of adjacent nodes with opposing signs.  The resulting $\xi$ may no longer be bounded in sup by $\frac{1}{2}$, but its $f$ value is unchanged, and the distance between the two pairs is now changed by at most 1, hence is 2, 3 or 4.  The value of $f(\xi)$ for all such pairs has been estimated, and none gives the optimum.

This reduces to the case of adding a singleton to a configuration of size 4.  There are four configurations of size 4 whose $P$ value may be paired with a singleton.  Many of the signs of the vertices are determined by the fact that a node of value 2 may not appear next to a node of the same sign.  The configurations are presented in Figure \ref{fig:epsilon_configurations}, with $\epsilon$ representing an ambiguous sign.

\begin{figure}
   \begin{tikzpicture}[node distance = 1.3cm, auto, place/.style = {circle, 
   thick, draw=blue!75, fill=blue!20}]
  \draw (0, 1.73205080757)--(-.5, .86602540378);
  \draw(0, 2*.86602540378) -- (-.5, 3* .86602540378);
  \draw(-.5, 3*.86602540378) -- (-1.5, 3*.86602540378);
   \node[place] at (-.5, .86602540378)  (lower)[label=$-1$]  {};
   \node[place] at (0, 1.73205080757) (lower)[label = $2$]   {};
   \node[place] at (-.5, 3*.86602540378) (lower)[label = $-1$]   {};
   \node[place] at (-1.5, 3*.86602540378) (lower) [label = $\epsilon$] {};
 \end{tikzpicture}\;\;
    \begin{tikzpicture}[node distance = 1.3cm, auto, place/.style = {circle, 
    thick, draw=blue!75, fill=blue!20}]
  \draw (0, 1.73205080757)--(1, 2*.86602540378);
  \draw(0, 2*.86602540378) -- (-.5, 3* .86602540378);
  \draw(-.5, 3*.86602540378) -- (-1.5, 3*.86602540378);
   \node[place] at (1, 2*.86602540378)  (lower)[label=$-1$]  {};
   \node[place] at (0, 1.73205080757) (lower)[label = $2$]   {};
   \node[place] at (-.5, 3*.86602540378) (lower)[label = $-1$]   {};
   \node[place] at (-1.5, 3*.86602540378) (lower) [label = $\epsilon$] {};
 \end{tikzpicture}\;\;
    \begin{tikzpicture}[node distance = 1.3cm, auto, place/.style = {circle, 
    thick, draw=blue!75, fill=blue!20}]
  \draw (0, 1.73205080757)--(-.5, .86602540378);
  \draw(0, 2*.86602540378) -- (-.5, 3* .86602540378);
  \draw(-.5, 3*.86602540378) -- (-1.5, 3*.86602540378);
  \draw(0, 2*.86602540378) --(1,2*.86602540378);
   \node[place] at (-.5, .86602540378)  (lower)[label=$-1$]  {};
   \node[place] at (0, 1.73205080757) (lower)[label = $2$]   {};
   \node[place] at (1,2* .86602540378) (lower)[label = $-1$]   {};
   \node[place] at (-1.5, 3*.86602540378) (lower) [label = $\epsilon$] {};
 \end{tikzpicture}
 
    \begin{tikzpicture}[node distance = 1.3cm, auto, place/.style = {circle, 
    thick, draw=blue!75, fill=blue!20}]
  \draw (0, 1.73205080757)--(-.5, .86602540378);
  \draw(0, 2*.86602540378) -- (-.5, 3* .86602540378);
  \draw(0, 2*.86602540378) --(1,2*.86602540378);
   \node[place] at (-.5, .86602540378)  (lower)[label=$-1$]  {};
   \node[place] at (0, 1.73205080757) (lower)[label = $2$]   {};
   \node[place] at (1,2* .86602540378) (lower)[label = $-1$]   {};
   \node[place] at (-.5, 3*.86602540378) (lower) [label = $-1$] {};
 \end{tikzpicture} 
 \caption{The size four configurations which have a node of height 2 and which can be paired with a singleton are pictured.  Adjacent nodes have opposite signs.  Placing 2 at position 0, the value of $\epsilon$ is forced to be 1 so that the singleton has sign $-1$ for the configuration to be mean 0, since otherwise its moment is not in the triangular lattice.}\label{fig:epsilon_configurations}
 \end{figure}
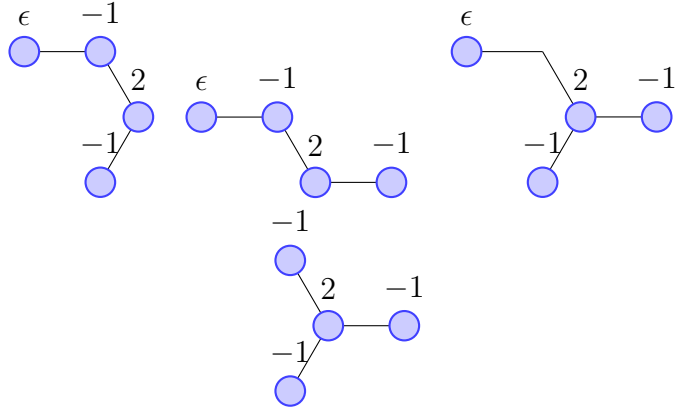

 After translation, assume the node of value 2 appears at 0.  The last figure may be ruled out, since to be mean 0, the singleton would have to appear at 0. In the remaining figures, both the node labeled 2 and the node labeled $\epsilon$ appear in the triangular lattice generated by $v_1, v_2$.  Thus to be mean 0, the singleton must appear translated by $v$ from the triangular lattice, with sign $-1$, and hence $\epsilon = 1$.  In each case, the mean 0 condition forces the singleton to be placed at distance 1 from the node labeled $\epsilon$, so none of these configurations obtains two connected components.  
 
 This reduces to the case of a single connected component. There is, up to symmetry, a single connected component on 3 nodes with value 2 at 0 and in $C^2(\sT)$, which has value $-1$ at $\pm v_1$. It was verified in SciPy that for $\nu_0 = 2$, $\nu_{\pm v_1} = -1$, $Q(\{0, \pm v_1\}, \nu) \geq 6.39$, and hence this configuration is not the optimum. The assignment of signs for the remaining configurations was restricted by requiring the node of height 2 to have value 2, and any adjacent nodes to have value $-1$.  There were 7 configurations on 5 nodes which have $P$ value at most 6 and with sign assignment in $C^2(\sT)$.  One of these was equivalent to the optimal configuration on six nodes arranged around a hexagon after subtracting the graph Laplacian at the point of height 2, the others were evaluated and do not give the optimum.  None of the configurations on 7 nodes could be given a sign assignment in $C^2(\sT)$.

\subsubsection*{ Case of two adjacent nodes of height 2.}  
  All configurations with two adjacent nodes of height 2 and all other nodes of height 1 which have $P$ value at most $6$ were enumerated.  All such configurations had at most 6 nodes.  No configurations with 4 or more nodes could be paired with a second connected component.  The only configuration on three nodes which could be paired with a singleton was, up to symmetry, given by the node 0 and two of its neighbors. By the above observation, after a sign change, one can assume $\nu_0 = 2$, $\nu_v = -2$, $\nu_{-v_1 + v} = -1$.  For this choice $P(\{0, v, -v_1+v\}, \nu) \geq 4.32$.  It follows that this configuration can only be paired with a singleton, which would have a forced location at distance less than 3 from the configuration (the location is forced to be $v_2$).  If the connected component has $\nu_0 = 2$, $\nu_v = -2$, then a second connected component cannot have nodes at distance 2 or more from each other, since $P(\{0, v\}, 2) + P(\{0, v_1\},1)\geq 4 + 2.59 > 6$, hence is either a pair of adjacent nodes, or a singleton.  No assignment of signs makes such a configuration in $C^2(\sT)$.  This reduces to the case of a single connected component.  Each of those configurations with a $P$ value at most 6 was tested for an assignment of signs that made the configuration in $C^2(\sT)$.  The only surviving configuration had 4 nodes and  the assignment shown in Figure \ref{fig:2_twos}.
 \begin{figure}
   \begin{tikzpicture}[node distance = 1.3cm, auto, place/.style = {circle, 
   thick, draw=blue!75, fill=blue!20}]
  \draw (0, 1.73205080757)--(-.5, .86602540378);
  \draw (0, 1.73205080757)--(1, 1.73205080757);
  \draw (1, 1.73205080757)--(1.5, .86602540378);
   \node[place] at (-.5, .86602540378)  (lower)[label=$-1$]  {};
   \node[place] at (1.5, .86602540378) (lower)[label=$1$]  {};
   \node[place] at (0, 1.73205080757) (lower)[label = $2$]   {};
   \node[place] at (1, 1.73205080757) (lower)[label = $-2$]   {};
 \end{tikzpicture}
 \caption{The only surviving configuration with two 2's on four nodes.  This did not achieve the optimum.}
 \label{fig:2_twos}
 \end{figure}
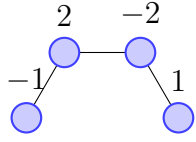
 This did not achieve the optimum.

\subsubsection{Face centered cubic lattice case} \label{fcc_section}

The optimum is shown to be achieved by $\nu^* = \delta_0 - \delta_{v_1}$, $\xi^* = g* \nu^*$ with $\|\xi^*\|_2^2 = 0.01867(5)$. The value $\gamma_{\fcc} = f(\xi^*) = 0.3623(9)$ was calculated by applying Lemma \ref{f_approx_lemma} with \begin{equation}R = \{n_1 v_1 + n_2 v_2 +n_3 v_3: |n_1|, |n_2|, |n_3| \leq 5\}.\end{equation}

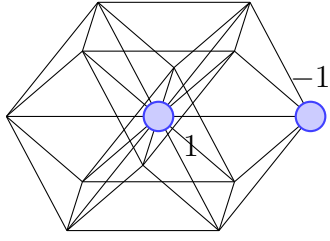
\begin{figure}
\begin{tikzpicture}[scale = 2, node distance = 1.3cm, auto, place/.style = {circle, 
thick, draw=blue!75, fill=blue!20}]
\draw (0,0)--(1,0);
\draw (0,0)--(0.5,0.433012701892219);
\draw (0,0)--(0.602062072615966,0.756710002993201);
\draw (0,0)--(-1,0);
\draw (0,0)--(-0.5,-0.433012701892219);
\draw (0,0)--(-0.602062072615966,-0.756710002993201);
\draw (0,0)--(0.5,-0.433012701892219);
\draw (0,0)--(0.397937927384034,-0.756710002993201);
\draw (0,0)--(-0.102062072615966,-0.323697301100982);
\draw (0,0)--(-0.5,0.433012701892219);
\draw (0,0)--(-0.397937927384034,0.756710002993201);
\draw (0,0)--(0.102062072615966,0.323697301100982);
\draw (1,0)--(0.5,0.433012701892219);
\draw (1,0)--(0.602062072615966,0.756710002993201);
\draw (1,0)--(0.5,-0.433012701892219);
\draw (1,0)--(0.397937927384034,-0.756710002993201);
\draw (0.5,0.433012701892219)--(0.602062072615966,0.756710002993201);
\draw (0.5,0.433012701892219)--(-0.5,0.433012701892219);
\draw (0.5,0.433012701892219)--(-0.102062072615966,-0.323697301100982);
\draw (0.602062072615966,0.756710002993201)--(0.102062072615966,0.323697301100982);
\draw (0.602062072615966,0.756710002993201)--(-0.397937927384034,0.756710002993201);
\draw (0.102062072615966,0.323697301100982)--(-0.397937927384034,0.756710002993201);
\draw (-1,0)--(-0.5,-0.433012701892219);
\draw (-1,0)--(-0.602062072615966,-0.756710002993201);
\draw (-1,0)--(-0.5,0.433012701892219);
\draw (-1,0)--(-0.397937927384034,0.756710002993201);
\draw (-0.5,-0.433012701892219)--(-0.602062072615966,-0.756710002993201);
\draw (-0.5,-0.433012701892219)--(0.5,-0.433012701892219);
\draw (-0.5,-0.433012701892219)--(0.102062072615966,0.323697301100982);
\draw (-0.602062072615966,-0.756710002993201)--(-0.102062072615966,-0.323697301100982);
\draw (-0.602062072615966,-0.756710002993201)--(0.397937927384034,-0.756710002993201);
\draw (-0.102062072615966,-0.323697301100982)--(0.397937927384034,-0.756710002993201);
\draw (0.5,-0.433012701892219)--(0.397937927384034,-0.756710002993201);
\draw (-0.5,0.433012701892219)--(-0.397937927384034,0.756710002993201);
\draw (0.5,-0.433012701892219)--(0.102062072615966,0.323697301100982);
\draw (-0.5,0.433012701892219)--(-0.102062072615966,-0.323697301100982);
\node[place] at (0,0)  (lower)[label=330:$1$]  {};
\node[place] at (1,0)  (lower)[label=$-1$]  {};

\end{tikzpicture}
\caption{The extremal configuration for the face centered cubic lattice.}\label{fig:fcc}
\end{figure}
It is more convenient to work with $\|\xi\|_2^2$ than $f(\xi)$.  By Lemma \ref{2_norm_lemma}, if $\|\xi\|_2^2 \geq \alpha$ with \begin{equation} 2\pi^2 \alpha\left(1 -\frac{\pi^2}{3}\alpha\right) > \gamma_{\fcc}, \qquad  \alpha = 0.01963\end{equation} then $f(\xi) >\gamma_{\fcc}$.

 Let $\xi$ be harmonic modulo 1, $\|\xi\|_\infty \leq \frac{1}{2}$. Let $\Delta \xi = \nu$. 
 \begin{lemma}
  If $\|\nu\|_{\infty} \geq 2$ then $\|\xi\|_2^2 \geq \frac{4}{12 \cdot 13} > 0.025 > \alpha$.
 \end{lemma}
\begin{proof}
 By Lemma \ref{2_norm_opt_lemma}, since $\deg(0) = 12$, $P'(\{0\}, 1) = \frac{1}{12 \cdot 13}$.  Within the interior of the domain, the objective function scales quadraticly, and hence $P'(\{0\}, 2) = \frac{4}{12 \cdot 13}$. Applying this translated to node $x$ where $|\nu_x| \geq 2$ implies the claim.
\end{proof}

Since the Green's function on a three dimensional lattice is not in $\ell^2$, it follows that the optimal $\nu$ is in $C^1(\sT)$, and hence has the same number of nodes with values $1$ and $-1$. 

\begin{lemma}
Suppose $|\supp \nu| \geq 4$ and let $z_1, z_2, z_3, z_4$ be four points in the support, two each with value $1, -1$.  The optimization problem $Q'(\{z_1, z_2, z_3, z_4\}, \nu)$ has value at least $\frac{2}{99} > \alpha$.  
 \end{lemma}
 \begin{proof}
Let the linear constraints be written as $\ell_i \cdot x = \nu_{z_i}$. Thus $\ell_i$ has value 12 at $z_i$ and value $-1$ at each of the 12 neighbors of $z_i$.  The optimum can be assumed to  not be achieved on the boundary, since this would exceed the claimed bound. By Lagrange multipliers, at the optimum, for some scalars $\lambda_1, \lambda_2, \lambda_3, \lambda_4$, $x = \lambda_1 \ell_1 + \lambda_2 \ell_2 +\lambda_3 \ell_3 + \lambda_4 \ell_4$.  Note that $\|\ell_i\|_2^2 = 12\cdot 13 = 156$.  When $\nu_{z_i} \neq \nu_{z_j}$, the maximum value of $\ell_i \cdot \ell_j$ is achieved when $z_i$ and $z_j$ are adjacent.  As they have 4 common neighbors, this maximum value is 20.  When $\nu_{z_i}$ and $\nu_{z_j}$ have the same sign, the maximum value is achieved when they differ by a rotation of $v_1 + v_2$. This maximum value is 2. Similarly, the minimum value of $\ell_i \cdot \ell_j$ when $\nu_{z_i} \neq \nu_{z_j}$ is $-2$ and the minimum value of $\ell_i \cdot \ell_j$ when $\nu_{z_i} = \nu_{z_j}$ is $-20$.  Note that the constraints may be written as
\begin{equation}
 \ell_i^t (\lambda_1 \ell_1 + \lambda_2 \ell_2 + \lambda_3 \ell_3 + \lambda_4 \ell_4) = \nu_{z_i}
\end{equation}
or
\begin{equation}
 156(I+A) \begin{pmatrix} \lambda_1\\ \lambda_2\\ \lambda_3\\\lambda_4 \end{pmatrix} = \begin{pmatrix} \nu_{z_1}\\ \nu_{z_2}\\ \nu_{z_3}\\ \nu_{z_4} \end{pmatrix}
\end{equation}
where $A$ has zeros on the diagonal and has row sums bounded in size by $\frac{42}{156}$.  Let $\lambda_i' = 156 \lambda_i $ and rewrite this as
\begin{equation}
 A\begin{pmatrix} \lambda_1'\\ \lambda_2'\\ \lambda_3'\\\lambda_4' \end{pmatrix} = \begin{pmatrix} \nu_{z_1} - \lambda_1'\\ \nu_{z_2} - \lambda_2' \\ \nu_{z_3} - \lambda_3' \\ \nu_{z_4} -\lambda_4'\end{pmatrix}.
\end{equation}
Thus $\max(|\nu_{z_i} - \lambda_i'|) \leq \frac{42}{156} \max (|\lambda_i'|)$.  Since $|\nu_{z_i}| = 1$, it follows that $\max(|\lambda_i'|) \leq 2$, and, hence $\max(|\nu_{z_i} - \lambda_i'|) \leq \frac{84}{156} < 1$ so $\lambda_i$ and $\nu_{z_i}$ have the same sign.

Write
\begin{align*}
 (\nu_{z_1}\ell_1 + \nu_{z_2} \ell_2 + \nu_{z_3} \ell_3 + \nu_{z_4}\ell_4)^t (\lambda_1 \ell_1 + \lambda_2 \ell_2 + \lambda_3 \ell_3 + \lambda_4 \ell_4) &= 4,
\end{align*}
and, by expanding the inner product on the right, express this as
\begin{align*}&a_1 \nu_{z_1} \lambda_1 + a_2 \nu_{z_2} \lambda_2 + a_3 \nu_{z_3} \lambda_3 + a_4 \nu_{z_4}\lambda_4\\
 &= a_1 |\lambda_1| + a_2 |\lambda_2| + a_3 |\lambda_3| + a_4 |\lambda_4| = 4,
\end{align*}
where 
\begin{equation*}
 a_i = \|\ell_i\|_2^2 + \sum_{j \neq i} \nu_{z_i}\nu_{z_j} \ell_i \cdot \ell_j.
\end{equation*}
 By the above considerations, $156 - 42 = 114 \leq a_i \leq 156 + 42 = 198$.
Since
\begin{align*}
 \|x\|_2^2 &= \lambda_1 \nu_{z_1} + \lambda_2 \nu_{z_2} + \lambda_3 \nu_{z_3} + \lambda_4 \nu_{z_4}\\
 &= |\lambda_1| + |\lambda_2| + |\lambda_3| + |\lambda_4|
\end{align*}
it follows that $\|x\|_2^2 \geq \frac{4}{198} = 0.\overline{02}>\alpha.$ 
\end{proof}

It follows that the optimum has $|\supp \nu| = 2$.  The following lemma reduces the search to a finite search.
\begin{lemma}
 Let $|\nu_0| = 1$ and $\nu_w = 0$ for $w$ such that $1\leq d(w,0)\leq 2$.  Let $S = \{w: d(w,0) \leq 2\}$.  Then $Q'(S, \nu) \geq 0.0125$. 
\end{lemma}
\begin{proof}
 This was verified in SciPy.
\end{proof}
 It follows that there may not be two points in the support of $\nu$ at graph distance greater than 6, or else the optimization problem could be applied at each point, and the 2-norm would be too large. This reduces the search to checking all configurations with two points in the support at graph distance at most 6.    The choice with adjacent points is the minimizer.

\subsubsection{$\Dfour$ tiling case} \label{d4_section}
The following optimization problems are used in the determination of the spectral parameters. In $\Dfour$, up to multiplication by a unit and reflection in the coordinate hyperplanes there is one element each of norm 1, 2, 3 and 4 in $\Dfour$.  Representatives are $1$, $1+i$, $1+i+j$ and $2$.
\begin{lemma}\label{opt_program_lemma}
 The following optimization problems have the corresponding values.  Let $\nu_0 = 1$. Then 
 \begin{equation}
 P'(\{0\}, 1) = \frac{1}{600} = 0.001\overline{6}.  
 \end{equation}
 
 Let $S = \{w: d(w,0) \leq 1\}$.  Let $\nu_0 = 1$ and $\nu_w = 0$ for $w$ such that $d(w,0) = 1$. Then 
 \begin{equation}
 Q'(S, \nu) \geq 0.00206.
 \end{equation}
 
Let $\nu_0 = 1$ and $\nu_w = 0$ for $1 \leq d(w,0) \leq 2$.  Then
\begin{equation}
 Q'(S, \nu) \geq 0.00233.
\end{equation}

 For $u \in \{1, 1+i, 1+i+j, 2\}$ and let $\nu_0 = 1$, $\nu_u = \pm 1$.  A lower bound for the program $Q'(\{0,u\}, \nu)$ in each case is given in the following table.

\begin{tabular}{|l|l|l|}
 \hline
 $u$ & $+1$ & $-1$\\
 \hline
 $1$ &$0.00357$ & $0.00312$\\
 $1+i$ & $0.00330$&$0.00336$ \\
 $1+i+j$ & $0.00332$& $0.00334$\\
 $2$ &$0.00332$& $0.00333$\\
 \hline
\end{tabular}

\end{lemma}
\begin{proof}
 The first value is the same as from Lemma \ref{2_norm_opt_lemma}.  The remaining values were determined in SciPy.
\end{proof}
 
Note that the first estimate of the Lemma implies that $P'(\{0\}, 2) \geq \frac{1}{150}$, since the objective function is quadratic.  This reduces to prevectors of height bounded by 1 in the calculations that follow.

\begin{lemma}\label{3_pt_support_lemma}
 If $\xi$ is harmonic modulo 1 on $\Dfour$ and $\nu = \Delta \xi$ has $|\supp \nu| \geq 3$, then $\|\xi\|_2^2 \geq \frac{3}{742} > 0.004043$.
\end{lemma}

\begin{proof}
 Let the points in the support of $\nu$ be $z_1, z_2, z_3$.  Then $\|\xi\|_2^2$ is bounded below by the value of the relaxed optimization program $P'(\{z_1, z_2, z_3\}, 1)$.  
 Applying Lagrange multipliers, the variable $x$ may be expressed as $\lambda_1 v_1 + \lambda_2 v_2 + \lambda_3 v_3$ where $v_1, v_2, v_3$ are the gradients of the constraint linear forms.  The linear constraints become $v_i^t (\lambda_1 v_1 + \lambda_2 v_2 + \lambda_3 v_3) = 1$ and 
 \begin{equation}
  \|x\|_2^2 = (\lambda_1 v_1 + \lambda_2 v_2 + \lambda_3 v_3)^t(\lambda_1 v_1 + \lambda_2 v_2 + \lambda_3 v_3) = \lambda_1 + \lambda_2 + \lambda_3.
 \end{equation}
 Since each $v_i$ has one entry 24 and 24 entries 1, $\|v_i\|_2^2 = 600$, and for $i \neq j$, $v_i^t v_j \leq 24+ 24+23 = 71$.
 Write the constraints as
 \begin{equation}
  600 (I + A) \begin{pmatrix} \lambda_1\\ \lambda_2 \\ \lambda_3 \end{pmatrix} = \begin{pmatrix} 1\\ 1\\1 \end{pmatrix}
 \end{equation}
with $A$ having 0's on the diagonal and row sums bounded in size by $\frac{142}{600}$.  Let $\lambda_i' = 600 \lambda_i$, so that
\begin{equation}
 A\begin{pmatrix} \lambda_1'\\ \lambda_2' \\ \lambda_3' \end{pmatrix} = \begin{pmatrix} 1-\lambda_1'\\ 1-\lambda_2'\\1-\lambda_3' \end{pmatrix}
\end{equation}
so that $ \max(|1-\lambda_i'|) \leq \frac{142}{600} \max(|\lambda_i'|)$.  This implies that $\max(|\lambda_i'|) \leq 2$ and thus $\max(|1-\lambda_i'|) \leq \frac{284}{600}$ so that each $\lambda_i > 0$.
Thus, summing constraints, $\lambda_1 + \lambda_2 + \lambda_3 \geq \frac{3}{742}.$\end{proof}

The proof of Theorem \ref{spectral_gap_calc_theorem} in the case of $\Dfour$ is as follows.
\subsubsection*{Case of $\gamma_{\Dfour, 0}$}
The extremal example is given by $\xi^* = g* \nu^*$ with $\nu^* = \delta_0 - \delta_{1}$. 

The 2-norm of $\xi^*$ was calculated by Parseval,
\begin{align*}
 \|\xi^*\|_2^2 &= \int_{(\bR/\zed)^4} \frac{2(1 - c(2y_1))}{g(y_1, y_2, y_3, y_4)^2}dy_1dy_2dy_3dy_4\\
 g(y_1, y_2, y_3, y_4) &= 24 - 2(c(y_1) + c(y_2) + c(y_3) + c(2y_4 - y_1-y_2-y_3)\\& + c(y_4) + c(y_4-y_1-y_2-y_3) + c(y_4-y_1) + c(y_4-y_2)\\& + c(y_4-y_3) + c(y_4 - y_1-y_2) + c(y_4-y_1-y_3)\\& + c(y_4-y_2-y_3) ) .
\end{align*}
This was calculated in SciPy, $\|\xi^*\|_2^2 = 0.0038397(3).$  By symmetry, $\|\xi^*\|_\infty^2 \leq \frac{1}{2} \|\xi^*\|_2^2$, and hence $\|\xi^*\|_4^4 \leq \frac{1}{2}\|\xi^*\|_2^4$. It follows that, for some $|\vartheta|< 1$, 
\begin{align*}
\gamma_{\Dfour,0}&= 2\pi^2 \|\xi^*\|_2^2 - \frac{\pi^4}{3}\|\xi^*\|_4^4 + \vartheta\frac{\pi^4}{3}\|\xi^*\|_4^4\\
 &= 2\pi^2 \|\xi^*\|_2^2 - \frac{\pi^4}{6} \|\xi^*\|_2^4 + \vartheta\frac{\pi^4}{6} \|\xi^*\|_2^4\\
 &= 0.075554+ \vartheta 0.00024.
\end{align*}
Thus,
\begin{align*}
 \Gamma_{\Dfour,0} = \frac{4}{\gamma_0} = 52.9428 + \vartheta 0.17.
\end{align*}

To verify that $\xi^*$ is extremal, suppose that $\xi$ is harmonic modulo 1, $\|\xi\|_\infty \leq \frac{1}{2}$ and $\Delta \xi = \nu$ is another candidate.  Since the Green's function is not in $\ell^2$ in dimension 4, it follows that $\nu \in C^1(\sT)$.  By Lemma \ref{2_norm_lemma}, to conclude that $\xi$ is not extremal, it suffices to conclude that $\|\xi\|_2^2 \geq \alpha$ with
\begin{equation}
 2\pi^2 \alpha\left(1 - \frac{\pi^2}{3}\alpha \right) \geq 0.075794, \qquad \alpha >0.0039. 
\end{equation}

Since the Green's function is not in $\ell^2$ in dimension 4, the prevector $\nu = \Delta \xi$ is in $C^1(\sT)$.    Note that $\|\nu\|_\infty = 1$, since the first optimization program in Lemma \ref{opt_program_lemma} can be applied where $|\nu_x| \geq 2$ and gives a value for the 2-norm which is too large.  Also, there are not 3 points in $\supp \nu$ by Lemma \ref{3_pt_support_lemma}. If the two points in the support of $\nu$ have distance at least 5, then the second optimization problem of Lemma \ref{opt_program_lemma} may be applied at each point, which makes the 2-norm too large.  Hence the two points in the support have graph distance at most 4.  One point may be taken to be 0.  The second point needs to be considered only up to multiplication by the 24 quaternion units and by reflection in the coordinate hyperplanes.  This leaves 24 candidates for the second point, which were checked exhaustively, the minimizer is $\nu_0$ and all other points had a 2-norm too large.

\subsubsection*{Case of $\gamma_{\Dfour,1}$}  By symmetry assume that the reflecting hyperplane is $\sP_1 =\{x \in \bR^4: x_1 + x_2 = 0\}$.  It is verified that the optimal prevector is $\nu^* = \delta_{(1, 0, 0, 0)}$, with reflection symmetry, so that $\nu^*(0, -1, 0,0) = -1$, and let $\xi^* = g*\nu^*$.  The 2-norm may be taken on the quotient by summing $\xi_x^2$ over points $x$ on one side of the hyperplane, including the hyperplane where $\xi$ vanishes, hence say
\begin{equation}
 \|\xi\|_2^2 = \sum_{x: x_1 + x_2 \geq 0} \xi_x^2.
\end{equation}
By Parseval,
\begin{equation}
 \|\xi^*\|_2^2 = 0.0022421(8).
\end{equation}
In this case
\begin{align*}
\gamma_{\Dfour,1}&= 2\pi^2 \|\xi_0\|_2^2 - \frac{\pi^4}{3}\|\xi^*\|_2^4 + \vartheta \frac{\pi^4}{3}\|\xi^*\|_2^4\\
 &=0.0440957 +\vartheta 0.00017.
\end{align*}
Thus,
\begin{align*}
 \Gamma_{\Dfour,1} = \frac{3}{\gamma_1} = 68.03486+ \vartheta 0.27.
\end{align*}

To check that $\xi^*$ is the optimizer, let $\xi$ be harmonic modulo 1, $\|\xi\|_\infty \leq \frac{1}{2}$, with reflection anti-symmetry in $\sP_1$ and let $\Delta \xi = \nu$.  It suffices to prove by Lemma \ref{2_norm_lemma} that  $\|\xi\|_2^2 \geq \alpha$ with 
\begin{equation}
 2\pi^2 \alpha\left(1 - \frac{\pi^2}{3}\alpha \right) \geq 0.04427, \qquad \alpha >0.00226. 
\end{equation}
If  $|\supp \nu| \geq 2$ in $\{x: x_1 + x_2 > 0\}$, then if two of the points in the support have distance at least 3 apart, the first optimization problem of Lemma \ref{opt_program_lemma} may be applied at each point.  Otherwise the last optimization problem may be applied.  In either case, the 2-norm is too large.  Thus there is a single point in the support, say $(x_1, x_2, x_3, x_4)$ of value 1.  The reflection point is $(-x_2, -x_1, x_3, x_4)$.  Let $z = x_1 + x_2$.  The 2-norm is, by Parseval,
\begin{equation}
 \int_{(\bR/\zed)^4} \frac{(1 - c(2z(y_1+y_2)))}{g(y_1, y_2, y_3, y_4)^2}dy_1dy_2dy_3dy_4.
\end{equation}
If $x$ has graph distance 3 or more from the boundary hyperplane $\sP_1$ then the third optimization program of Lemma \ref{opt_program_lemma} may be applied to show that the 2-norm is too large.  It now follows by checking case-by-case that the minimizer is $z = 1$, which is $\nu^*$.

\subsubsection*{Case of $\gamma_{\Dfour,2}$}  By symmetry  assume that the reflecting hyperplanes are $\sP_1 = \{x \in \bR^4: x_1 + x_2 = 0\}$ and $\sP_2 = \{x \in \bR^4: x_1 - x_2 = 0\}$. The optimizing prevector is $\nu^* = \delta_{(1,0,0,0)}$ and $\xi^* = g* \nu^*$.  By reflection anti-symmetry, 
\begin{equation}
 \nu^*(0, -1, 0, 0) = -1, \qquad \nu^*(-1, 0, 0, 0) = 1, \qquad \nu^*(0, 1, 0, 0) = -1.
\end{equation}
The 2-norm is $\|\xi^*\|_2^2 = 0.0019800(3).$ Calculating as in the case of $\gamma_{\Dfour, 1}$,
\begin{align*}
\gamma_{\Dfour,2}&= 2\pi^2 \|\xi^*\|_2^2 - \frac{\pi^4}{3}\|\xi^*\|_2^4 + \vartheta \frac{\pi^4}{3}\|\xi^*\|_2^4\\
 &=0.0389569 +\vartheta 0.00013.
\end{align*}
Thus,
\begin{align*}
 \Gamma_{\Dfour,2} =\frac{2}{\gamma_{\Dfour, 2}} =  51.3393 + \vartheta 0.17.
\end{align*}

To verify that $\xi^*$ is extremal, let $\xi$ be harmonic modulo 1 with reflection anti-symmetry in $\sP_1$ and $\sP_2$ and let $\nu = \Delta \xi$.  To rule out that $\xi$ is extremal it suffices to check by Lemma \ref{2_norm_lemma} that $\|\xi\|_2^2 \geq \alpha$ with 
\begin{equation}
 2\pi^2 \alpha\left(1 - \frac{\pi^2}{3}\alpha \right) \geq 0.0391, \qquad \alpha >0.002. 
\end{equation}
The case of two points in the support modulo reflections is ruled out as before.  Suppose the point in the support is $(x_1, x_2, x_3, x_4)$.  This point may have not have distance at least 2 from both hyperplanes, or else the second optimization problem of Lemma \ref{opt_program_lemma} may be applied to show that the 2-norm is too large.  Hence, $\min(|x_1 + x_2|, |x_1 -x_2|) = 1$, say, by symmetry, $x_1 + x_2 = 1$ and $x$ differs by $(1,1,0,0)$ from its reflection in $\sP_1$.  If $x$ has graph distance 3 or more from $\sP_2$ then the  optimization program which enforces $\Delta \xi (0,0,0,0)= 1$, $\Delta \xi (1,1,0,0) = -1$ and $\Delta \xi(v) = 0$ if $d(v, \{(0,0,0,0), (1,1,0,0)\}) \leq 2$ can be applied, which has minimum 2-norm  $0.0041780(9)$.  This is a lower bound for twice the 2 norm of $\xi$ modulo reflections, and is too large.  Hence $x$ has graph distance at most 2 from $\sP_2$, so $x_1 - x_2$ is either 1 or 2.  The case 1 is the minimizer. 

\subsubsection*{Case of $\gamma_{\Dfour, 3}$}  By symmetry assume that the reflecting hyperplanes are $\sP_1, \sP_2, \sP_3$, with
\begin{equation}
 \sP_3 = \{x \in \bR^4: x_3 + x_4 = 0\}.
\end{equation}
It is shown that the minimizer is $\nu^* = \delta_{(1,0,1,0)}$ with $\xi^* = g*\nu^*$, $\|\xi^*\|_2^2 = 0.0018737(9)$. The corresponding value of $\alpha$ to rule out other configurations is $\alpha = 0.00189$.  Arguing as for $\gamma_{\Dfour, 1}$ and $\gamma_{\Dfour, 2}$,
\begin{equation*}
 \gamma_{\Dfour,3} = 0.036873324 +\vartheta 0.00012
\end{equation*}
Thus,
\begin{equation*}
 \Gamma_{\Dfour,3} = \frac{1}{\gamma_{\Dfour, 3}} = 27.1201 + \vartheta 0.084
\end{equation*}

Arguing as above, we may assume that $|\supp \nu^*| = 1$, and that the point $x = (x_1, x_2, x_3, x_4)$ has distance 1 to one hyperplane, say, without loss, that $x_1 + x_2 = 1$.  Also, the distance to the next closest hyperplane is at most 2, so say $x_1 - x_2 = 1$ or $x_1 - x_2 = 2$.  Now considering $\Delta \xi$ at $(0,0,0,0), (1,1,0,0), (1,-1,0,0), (2,0,0,0)$ or $(0,0,0,0), (1,1,0,0), (2,-2,0,0), (3,-1,0,0)$ and their distance 2 neighborhood implies that the distance to $\sP_3$ is at most 2.  The three possibilities are considered, and $\nu^*$ gives the optimum. 

\subsubsection*{Case of $\gamma_{\Dfour,4}$}
   The minimizer is $\nu^* = \delta_{(1,0,1,0)}$ with $\xi^* = g*\nu^*$, $\|\xi^*\|_2^2 = 0.0018170(7)$. The value of $\alpha$ to rule out other configurations in this case is $\alpha = 0.0018281$. This obtains
\begin{equation*}
 \gamma_{\Dfour,4} = 0.0357604+ \vartheta 0.00011.
\end{equation*}
      The above considerations reduce to the case where $\nu$ is supported at a single point, with distance 1 from $\sP_1$, and distance at most 2 from $\sP_2$ and $\sP_3$.  Arguing similarly to the case of three hyperplanes shows that the distance to $\sP_4$ is also at most 2, which reduces to a finite check.  The best case is $\nu^*$.

\section{The spectral parameters of the $\zed^d$ tiling} \label{zd_section}

This section evaluates the spectral factor of the $\zed^d$ lattice asymptotically, proving Theorem \ref{asymptotic_theorem}.

When $d \geq 3$, the Green's function on $\zed^d$ may be recovered from its Fourier transform via Fourier inversion,
\begin{equation}
 g_0(x) = \frac{1}{2d} \int_{(\bR/\zed)^d} \frac{e(x\cdot y)}{1-\frac{1}{d}(c(y_1) + \cdots + c(y_d) )}dy.
\end{equation}
The following lemma is useful in studying this integral evaluation asymptotically.
\begin{lemma}
 Let $X_1, X_2, ..., X_d$ be i.i.d. random variables on $[-1, 1]$ with distribution 
 \begin{equation}
  \Prob(X_1 \leq a) = \meas\left\{0 \leq t \leq 1: \cos(2\pi t) \leq a\right\}.
 \end{equation}
Let $X = \frac{1}{d}(X_1 + X_2 + \cdots + X_d).$  For $0 \leq \delta \leq 1$, 
\begin{equation}
 \Prob(X > 1-\delta) \leq \min\left(e^{-\frac{d(1-\delta)^2}{2}}, \left(\frac{\pi e\delta}{4}\right)^{\frac{d}{2}} \right) .
\end{equation}
The two bounds are equal for $\delta = 0.27819(3)=:\xi$.
\end{lemma}
\begin{proof}
 The first bound follows from Chernoff's inequality, since $\E[X_1] = 0$ and $\E[X_1^2] = \frac{1}{2}$, so that one may take $\sigma = \sqrt{\frac{d}{2}}$ and $\lambda = \sqrt{2d}(1-\delta)$.
 For the second bound, use $1-\cos 2\pi t \geq 8t^2$ for $|t| \leq \frac{1}{2}$. Thus, estimating with the Euclidean volume of a ball of radius $r$ in $d$ dimensions, 
 $\vol B_r(0) = \frac{r^d \pi^{\frac{d}{2}}}{\Gamma\left(\frac{d}{2}+1 \right)}$,
 \begin{equation}
 \Prob(X > 1-\delta) \leq \meas\left(\ut \in \left[-\frac{1}{2}, \frac{1}{2}\right]^d: \|\ut\|_2^2 \leq \frac{\delta d}{8} \right)\leq \frac{(\pi \delta d)^{\frac{d}{2}}}{8^{\frac{d}{2}}\Gamma\left(\frac{d}{2}+1 \right)}.
\end{equation}
Now use 
\begin{equation}
 \Gamma\left(\frac{d}{2}+1\right) \geq \left(\frac{d}{2e} \right)^{\frac{d}{2}}
\end{equation}
which is valid for $d \geq 2$.
\end{proof}

The following lemma estimates $\|\xi\|_2^2$ asymptotically when $\nu$ is a singleton.  This example controls the mixing time for all $d$ sufficiently large.

\begin{lemma}\label{Greens_fn_evaluation_d}
 Let $d \geq 5$, let $0 \leq k \leq d$, and let $\nu:\zed^d \to \zed$ which has reflection anti-symmetry in the first $k$ coordinate hyperplanes $\sP_1, ..., \sP_k$. Let $\fS_k$ be the group of reflections in $\sP_1, ..., \sP_k$ and suppose that modulo $\fS_k$, $\nu$ is a point mass. Let $\xi = g*\nu$.  Then as $d \to \infty$,
 \begin{equation}
  \|\xi\|_{2, \zed^d/\fS_k}^2 = \frac{1}{4d^2}\left(1 + \frac{3}{2d} + O_k\left(d^{-2}\right)\right).
 \end{equation}

\end{lemma}

\begin{proof}
 When $d \geq 5$ the Green's function is in $\ell^2(\zed^d)$.  The $\ell^2$ norm is
 \begin{equation}
  \left\|g\right\|_2^2 = \frac{1}{4d^2} \int_{(\bR/\zed)^d} \frac{dy}{\left(1 - \frac{1}{d}\left(c(y_1) + \cdots + c(y_d) \right)\right)^2}.
 \end{equation}
 Note that $\frac{1}{d} \left(c(y_1) + \cdots + c(y_d) \right)$ is the Fourier series of the measure $\mu$ of simple random walk on $\zed^d$.

When $\nu$ has reflection symmetry in $k$ hyperplanes and is supported at a single point $\ua$ in the quotient space, the 2 norm of $g*\nu$ in the quotient space is
\begin{align*}
 \left\|g*\nu\right\|_{2, \zed^d/\fS_k}^2 &= \frac{1}{4d^2} \int_{(\bR/\zed)^d} \frac{\frac{1}{2^k}\left|\prod_{j=1}^k (e(a_j y_j)-e(-a_jy_j)) \right|^2}{\left(1 - \frac{1}{d}\left(c(y_1) + \cdots + c(y_d) \right)\right)^2}dy\\
 &=\frac{1}{4d^2} \int_{(\bR/\zed)^d} \frac{\prod_{j=1}^k (1-c(2a_jy_j))}{\left(1 - \frac{1}{d}\left(c(y_1) + \cdots + c(y_d) \right)\right)^2}dy.
\end{align*}
By symmetry of the random walk,
\begin{equation}\label{2_norm_formula}
 \left\|g*\nu\right\|_{2, \zed^d/\fS_k}^2 =\frac{1}{4d^2} \int_{(\bR/\zed)^d} \frac{\prod_{j=1}^k (1-e(2a_jy_j))}{\left(1 - \frac{1}{d}\left(c(y_1) + \cdots + c(y_d) \right)\right)^2}dy.
\end{equation}

Use the formula,
\begin{equation}
 \frac{1}{(1-x)^2} = 1 + 2x + 3x^2 + \cdots + nx^{n-1} + \frac{nx^n}{1-x} + \frac{x^n}{(1-x)^2}
\end{equation}
with $x = \frac{1}{d}\left(c(y_1) + \cdots + c(y_d) \right)$.

To estimate $\|g\|_{2}^2$, write this as
\begin{align*}
 \|g\|_{2}^2 &= \frac{1}{4d^2} \left(1 + 3 \int x^2 + 4 \int \frac{x^4}{1-x} + \int \frac{x^4}{(1-x)^2} \right)\\
 &= \frac{1}{4d^2} \left(1 + \frac{3}{2d} +   \int \frac{4x^4}{1-x} + \frac{x^4}{(1-x)^2} \right).
\end{align*}
To estimate the integrals, by symmetry, pair $x$ and $-x$, so that the integrals become
\begin{align*}
 &-\int_0^1 c^4\left(\frac{8}{1-c^2} + \frac{2 + 2c^2}{(1-c^2)^2} \right)d\Prob(x \geq c) \\
 &=  \int_0^1 \frac{d}{dc}\left(\frac{c^4 (10-6c^2)}{(1-c^2)^2} \right) \Prob(x\geq c) dc\\
 &=  \int_0^1 \left(\frac{40c^3 -36 c^5}{(1-c^2)^2} + \frac{40c^5 - 24c^7}{(1-c^2)^3} \right)\Prob(x\geq c) dc\\
 & \leq  \int_0^{1-\xi} \left( \frac{40 c^3  + 12 c^7}{(1-c^2)^3} \right)e^{-\frac{dc^2}{2}} dc \\ &+\int_0^{\xi} \left( \frac{40(1-c)^3  + 12(1-c)^7}{c^3 (2-c)^3} \right)\left(\frac{\pi e c}{4} \right)^{\frac{d}{2}} dc.
\end{align*}
In the first integral, bound $\frac{1}{1-c^2} \leq \frac{1}{2\xi - \xi^2}$, then extend the integrals to $\infty$ to obtain a bound of $O(d^{-2})$.  The second integral is exponentially small in $d$.

Also, estimate, for $\ua \neq 0$, taking absolute values in the final integral,
\begin{align*}
& \frac{1}{4d^2} \int_{(\bR/\zed)^d}\frac{e(2\sum_{j=1}^k a_jy_j)}{\left(1-\frac{1}{d}\left(c(y_1)+\cdots + c(y_d) \right) \right)^2}dy \\
&= \frac{3}{4d^2} \mu^{*2}(2a_1 e_1 + \cdots +2a_k e_k) + O\left(\frac{1}{4d^2} \int \frac{x^4}{(1-x)^2} \right)\\
&= \frac{3}{16 d^4} \one(\|\ua\|_1=1) + O\left(\frac{1}{d^4} \right).
\end{align*}
By expanding the numerator of (\ref{2_norm_formula}), this implies that
\begin{equation}
 \|g*\nu\|_{2, \zed^d/\fS_k}^2 = \frac{1}{4d^2} \left(1 + \frac{3}{2d} + O_k(d^{-2})\right).
\end{equation}
\end{proof}

The following lemma evaluates $\|\xi\|_2^2$ asymptotically when the support of $\nu$ is larger.

\begin{lemma}\label{Greens_fn_evaluation_d_larger_support}
 Let $d \geq 5$.  Let $\nu = \delta_0 - \delta_{\ua}$ for some $\ua \neq 0 \in \zed^d$, and let $\xi = g*\nu$.  As $d \to \infty$,
  \begin{equation}
  \|\xi\|_{2}^2 = \left\{ \begin{array}{lll} \frac{1}{2d^2}\left(1 + \frac{1}{2d} + O\left(d^{-2}\right)\right) && \|\ua\|_1 = 1\\ \frac{1}{2d^2}\left(1 + \frac{3}{2d} + O\left(d^{-2}\right)\right) && \|\ua\|_1>1\end{array}\right..
 \end{equation}

\end{lemma}
\begin{proof}
 As in the previous lemma, let $\hat{\mu}(y) = \frac{1}{d}(c(y_1) + \cdots + c(y_d))$.  Then, by Parseval, estimating the error as above,
 \begin{align*}
  &\|g * \nu\|_2^2 = \frac{1}{4d^2} \int_{(\bR/\zed)^d} \frac{|\hat{\nu}(y)|^2}{(1-\hat{\mu}(y))^2}dy\\
  &= \frac{1}{4d^2}\int_{(\bR/\zed)^d} |\hat{\nu}(y)|^2 \left(1 + 2\hat{\mu}(y) + 3 \hat{\mu}(y)^2 + 4 \hat{\mu}(y)^3 \right)dy + O(d^{-4}).
 \end{align*}
Let $\check{\nu}(x) = \nu(-x)$.  The integral evaluates to, by Parseval,
\begin{equation*}
 \int_{(\bR/\zed)^d} |\hat{\nu}(y)|^2 \hat{\mu}(y)^j dy = \nu * \check{\nu} * \mu^{*j}(0).
\end{equation*}
Since $\nu * \check{\nu}(0) = \|\nu\|_2^2$ and \begin{equation*}\mu^{*0}(0)+ 2\mu(0) + 3 \mu^{*2}(0) + 4 \mu^{*3}(0) = 1 + \frac{3}{2d},\end{equation*} the $\nu*\check\nu(0)$ terms contribute $\|\nu\|_2^2 \left(1 + \frac{3}{2d}\right)$.  For $\|\ua\|_1 = 1$, \begin{equation}\mu^{*0}(\ua) + 2\mu(\ua) + 3\mu^{*2}(\ua) + 4 \mu^{*3}(\ua) = \frac{1}{d} + O\left(\frac{1}{d^2} \right),\end{equation} while for $\|\ua\|_1 > 1$, the sum is $O(d^{-2})$.  Thus, for $\nu = \delta_0 - \delta_{\ua}$, \begin{equation}\nu * \check{\nu} = 2 \delta_0 - \delta_{\ua} - \delta_{-\ua}\end{equation} and, 
\begin{equation}
 \|g*\nu\|_2^2 = \frac{2}{4d^2} \left(1 + \frac{3}{2d} - \one(\|\ua\|=1)\frac{1}{d} + O\left(\frac{1}{d^2}\right) \right).
\end{equation}
\end{proof}

The results obtained by integration are to be compared with the following lower bounds for $\|\xi\|_2^2$ obtained from a convex optimization program.
\begin{lemma}\label{program_lower_bounds}
 Let $\xi = g*\nu$ be a function on $\zed^d$, $d \geq 2$ with reflection anti-symmetry in the first $k$ coordinate hyperplanes. Let the corresponding reflection group be $\fS_k$.  Consider $\xi$ and $\nu$ to be anti-symmetric functions on the quotient of $\zed^d/\fS_k$. The following bounds hold for $\|\xi\|_{2, \zed^d/\fS_k}^2$.
 
 \begin{equation}
  \|\xi\|_{2, \zed^d/\fS_k}^2 \geq \left\{ \begin{array}{lll}\frac{1}{4d^2 + 2d} && |\supp \nu| \geq 1\\ \frac{2}{4d^2 + 6d} && |\supp \nu| \geq 2, u, u+e_j \in \supp \nu\\
  \frac{2}{4d^2 + 2d+1} && |\supp \nu| \geq 2, u, u+2e_j \in \supp \nu\\
  \frac{2}{4d^2 + 2d+2} && |\supp \nu| \geq 2, u, u+e_i +e_j \in \supp \nu\\
  \frac{2}{4d^2 + 2d} && |\supp \nu| \geq 2, u, w \in \supp \nu, d(u,w)\geq 3\end{array}\right..
 \end{equation}

\end{lemma}
\begin{proof}
 Denote $P(\{0\}, 1)$ the optimization program
 \begin{align*}
\text{minimize:} &\qquad  \sum_{d(w, \{0\}) \leq 1} x_w^2\\
\text{subject to:} &\qquad 2dx_0 + \sum_{i=1}^d x_{e_i} + x_{-e_i} \geq 1\\ &\qquad |x_w| \leq \frac{1}{2},
\end{align*}
 which is a lower bound for the first quantity.  This program has value $\frac{1}{(2d)(2d+1)}$, since the optimum occurs at an interior point, and is achieved by $x_0 = \frac{1}{2d+1}$, $x_w = \frac{1}{(2d)(2d+1)}$ for $d(w,0)=1$.  

 In the last case, two translated copies of $P(\{0\}, 1)$ may be applied, one at each point in the support.

In the remaining cases, a lower bound for $\|\xi\|_{2, \zed^d/\fS_k}^2$ is given by setting, for $u = e_1, 2e_1, e_1 + e_2$, $\nu_0 = \nu_u = 1$, and calculating  $P(\{0, u\}, 1)$
\begin{align*}
\text{minimize:} &\qquad  \sum_{d(w, \{0, u\}) \leq 1} x_w^2\\
\text{subject to:} &\qquad 2dx_0 + \sum_{i=1}^d x_{e_i} + x_{-e_i} \geq 1\\
&\qquad 2d x_{u} + \sum_{i=1}^d x_{u + e_i} + x_{u-e_i}\geq 1\\
& \qquad |x_w| \leq \frac{1}{2}.
\end{align*}
The optimum in this case is achieved at an interior point since the values on the boundary are at least $\frac{1}{4}$, which exceeds the claimed bound. At an interior point, by Lagrange multipliers the optimum takes the form
 \begin{equation}
  x = \lambda_1 v_1 + \lambda_2 v_2
 \end{equation}
where $v_1$ and $v_2$ are the gradients of the two linear constraints. The linear system $v_1^t (\lambda_1 v_1 + \lambda_2 v_2) = 1$, $v_2^t (\lambda_1 v_1 + \lambda_2 v_2) = 1$ is symmetric in $\lambda_1, \lambda_2$ and has a unique solution with $\lambda=\lambda_1 = \lambda_2 = \frac{1}{\|v_1\|_2^2 + v_1^t v_2}$.  Thus
\begin{equation}
 \|x\|_2^2 = 2\lambda^2 (\|v_1\|_2^2 + v_1^t v_2) = \frac{2}{\|v_1\|_2^2 + v_1^t v_2}. 
\end{equation}
Since $\|v_1\|_2^2 = 4d^2 +2d$ and $v_1^t v_2$ has value $4d, 1, 2$ in the three cases considered, the claim follows.

\end{proof}

\begin{lemma}\label{program_lower_bounds_larger}
 Let $\xi = g*\nu$ be a function on $\zed^d$, $d \geq 5$ satisfying $|\supp \nu| \geq 3$.  Then
 \begin{equation}
  \|\xi\|_2^2 \geq \frac{3}{4d^2 + 10 d}.
 \end{equation}
\end{lemma}
\begin{proof}
 After translation, suppose that $0, u_1, u_2$ are in the support.  A lower bound for the 2 norm is given by the value of the optimization program $P(\{0, u_1, u_2\}, 1)$
  \begin{align*}
\text{minimize:} &\qquad  \sum_{d(w, \{0, u_1, u_2\}) \leq 1} x_w^2\\
\text{subject to:} &\qquad 2dx_0 + \sum_{i=1}^d x_{e_i} + x_{-e_i} \geq 1\\ & \qquad 2 dx_{u_1} + \sum_{i=1}^d x_{u_1 + e_i} + x_{u_1 - e_i} \geq 1\\ & \qquad 2 dx_{u_2} + \sum_{i=1}^d x_{u_2 + e_i} + x_{u_2 - e_i}\geq 1.
\end{align*}
Let $x$ be the set of variables and write the constraints as $v_1^t x \geq 1$, $v_2^t  x \geq 1$, $v_3^t x \geq 1$. By Lagrange multipliers, the optimum occurs at $x = \lambda_1 v_1 + \lambda_2 v_2 + \lambda_3 v_3$.  Since the distance 1 neighborhoods of $0, u_1,u_2$ pairwise overlap in at most 2 points, each neighborhood has some variable not shared by the others.  Since the optimum occurs with all variables non-negative, $\lambda_1, \lambda_2, \lambda_3 \geq 0$.  

We have $\|v_i\|_2^2 = 4d^2 + 2d$ and for $i \neq j$, $v_i^t v_j \leq 4d$.  Adding the three constraints, 
\begin{equation}
 (v_1^t + v_2^t + v_3^t)(\lambda_1 v_1 + \lambda_2 v_2 + \lambda_3 v_3) \geq 3.
\end{equation}
Hence $(\lambda_1 + \lambda_2 + \lambda_3) \geq \frac{3}{4d^2 + 10d}.$  Since
\begin{align*}
 \|x\|_2^2 &= \left(\lambda_1 v_1^t + \lambda_2 v_2^t + \lambda_3 v_3^t\right)(\lambda_1 v_1 + \lambda_2 v_2 + \lambda_3 v_3)\\& \geq
 \lambda_1 v_1^t (\lambda_1 v_1 + \lambda_2 v_2 + \lambda_3 v_3) + \lambda_2 v_2^t (\lambda_1 v_1 + \lambda_2 v_2 + \lambda_3 v_3) \\& \qquad+ \lambda_3 v_3^t (\lambda_1 v_1 + \lambda_2 v_2 + \lambda_3 v_3) \\ & \geq \lambda_1 + \lambda_2 + \lambda_3 \\&\geq \frac{3}{4d^2 + 10 d}.
\end{align*}

\end{proof}

Combining the estimates, it is now possible to prove Theorem \ref{asymptotic_theorem}.

\begin{proof}[Proof of Theorem \ref{asymptotic_theorem}]
 In determining $\gamma_{\zed^d}$, $\nu = \Delta \xi$ is $C^1$, and hence $|\supp \nu| \geq 2$.  Using $1- c(\xi_x) = 2\pi^2 \xi_x^2 + O(\xi_x^4)$ it follows that $f(\xi) = 2\pi^2 \|\xi\|_2^2 + O(\|\xi\|_2^4)$. By Lemma \ref{program_lower_bounds_larger}, if $|\supp \nu| \geq 3$ then $\|\xi\|_2^2 \geq \frac{3}{4d^2 + 10d}$.  Combining with Lemma \ref{Greens_fn_evaluation_d_larger_support}, for all $d$ sufficiently large the optimum is achieved by $\nu = \delta_0 - \delta_{e_1}$ with
 $\|\xi\|_2^2 = \frac{1}{2d^2}\left(1 + \frac{1}{2d} + O\left(d^{-2}\right)\right)$.  Hence
 \begin{equation}
  \gamma_{\zed^d} = \frac{\pi^2}{d^2}\left(1 + \frac{1}{2d} + O\left(d^{-2}\right)\right).
 \end{equation}

By Lemma \ref{program_lower_bounds}, it follows that if $|\supp \nu| \geq 2$ for the extremal prevector, then $\|g*\nu\|_2^2 \geq \frac{1}{2d^2} + O(d^{-3})$. Thus, asymptotically in $d$, the extremum is achieved with $\nu$ a point mass. 
Approximating $1-c(\xi_x) = 2\pi^2 \xi_x^2 + O(\xi_x^4)$, it follows from Lemma \ref{Greens_fn_evaluation_d} that, as $d \to \infty$, for each $j$,
\begin{equation}
 \gamma_{\zed^d,j} = \frac{\pi^2}{2d^2} \left(1 + \frac{3}{2d} + O_j(d^{-2})\right)
\end{equation}
and, uniformly in $j$, by the first estimate of Lemma \ref{program_lower_bounds},
\begin{equation}
 \gamma_{\zed^d, j} \geq \frac{\pi^2}{2d^2 + d}(1 + O(d^{-2})).
\end{equation}
It follows that, for all $j$,
\begin{equation}
 \Gamma_j \leq \frac{(d-j)(2d^2 + d + O(1))}{\pi^2} 
\end{equation}
and for each fixed $j$,
\begin{equation}
 \Gamma_j = \frac{(d-j) (2d^2 -3d + O_j(1))}{\pi^2}.
\end{equation}
In particular, $\Gamma = \Gamma_0 = \frac{2d^3 -3d^2 + O(d)}{\pi^2}$ for all $d$ sufficiently large.

\end{proof}

\bibliographystyle{plain}

\begin{thebibliography}{1}

\bibitem{ADMR10}
Azimi-Tafreshi, N., H. Dashti-Naserabadi, S. Moghimi-Araghi, and Philippe Ruelle. 
\newblock ``The Abelian sandpile model on the honeycomb lattice." 
\newblock \emph{Journal of Statistical Mechanics: Theory and Experiment} 2010, no. 02 (2010): P02004.

\bibitem{BTW88}
Bak, P., C. Tang, and K. Wiesenfeld. 
\newblock ``Self-organized criticality." 
\newblock \emph{Physical Review A} 38.1 (1988): 364.

\bibitem{BIP93}
Brankov, J. G., E. V. Ivashkevich, and V. B. Priezzhev. 
\newblock ``Boundary effects in a two-dimensional Abelian sandpile." 
\newblock \emph{Journal de Physique} I 3.8 (1993): 1729-1740.

\bibitem{DFF03}
Dartois, Arnaud, Francesca Fiorenzi, and Paolo Francini.
\newblock ``Sandpile group on the graph $D_n$ of the dihedral group." 
\newblock \emph{European Journal of Combinatorics} 24.7 (2003): 815-824.


\bibitem{D89}
Dhar, Deepak. 
\newblock ``Self-organized critical state of sandpile automaton models." 
\newblock \emph{Physical Review Letters} 64.14 (1990): 1613.


\bibitem{DM92}
Deepak Dhar, and Majumdar, Satya N. 
\newblock ``Equivalence between the Abelian sandpile model and the $q\to 0$ limit of the Potts model." 
\newblock \emph{Physica A: Statistical Mechanics and its Applications} 185.1-4 (1992): 129-145.


\bibitem{DS10}
Dhar, Deepak, and Sadhu, Tridib. 
\newblock ``Pattern formation in growing sandpiles with multiple sources or sinks.'' \emph{J. Stat. Phys.} 138 (2010), no. 4-5, 815--837. 

\bibitem{D96}
Diaconis, Persi. 
\newblock ``The cutoff phenomenon in finite Markov chains." \newblock \emph{Proceedings of the National Academy of Sciences} 93.4 (1996): 1659-1664.

\bibitem{DS87}
Diaconis, Persi, and Mehrdad Shahshahani. 
\newblock ``Time to reach stationarity in the Bernoulli--Laplace diffusion model." 
\newblock \emph{SIAM Journal on Mathematical Analysis} 18.1 (1987): 208-218.

\bibitem{D88}
Diaconis, Persi. 
\newblock ``Group representations in probability and statistics.'' \emph{Lecture Notes-Monograph Series} 11 (1988): i-192.

\bibitem{FLW10}
Fey, Anne, Lionel Levine, and David B. Wilson.
\newblock ``Driving sandpiles to criticality and beyond." 
\newblock \emph{Physical review letters} 104.14 (2010): 145703.

\bibitem{G16}
Gamlin, Samuel. 
\newblock \emph{Boundary conditions in Abelian sandpiles.}
\newblock Diss. University of Bath, 2016.

\bibitem{HS19}
Hough, Robert and Hyojeong Son.
\newblock ``Cut-off for sandpiles on tiling graphs.'' 

\bibitem{HJL17}
Hough, Robert, Dan Jerison, and Lionel Levine. 
\newblock ``Sandpiles on the square lattice." 
\newblock \emph{Communications Math Phys}, to appear (2017).


\bibitem{I94}
Ivashkevich, E. V. 
\newblock ``Boundary height correlations in a two-dimensional Abelian sandpile." 
\newblock \emph{Journal of Physics A: Mathematical and General} 27.11 (1994): 3643--3653.


\bibitem{JPR06}
Jeng, Monwhea, Geoffroy Piroux, and Philippe Ruelle. 
\newblock ``Height variables in the Abelian sandpile model: scaling fields and correlations."
\newblock \emph{Journal of Statistical Mechanics: Theory and Experiment} 2006.10 (2006).

\bibitem{JLP15}
Jerison, Daniel C., Lionel Levine, and John Pike. 
\newblock ``Mixing time and eigenvalues of the abelian sandpile Markov chain." 
\newblock arXiv:1511.00666 (2015).



\bibitem{KW16}
Kassel, Adrien, and David B. Wilson. 
\newblock``The looping rate and sandpile density of planar graphs." 
\newblock \emph{The American Mathematical Monthly} 123.1 (2016): 19-39.

\bibitem{KS16}
Kalinin, Nikita, and Mikhail Shkolnikov. 
\newblock ``Tropical curves in sandpiles." 
\newblock \emph{Comptes Rendus Mathematique} 354.2 (2016): 125-130.

\bibitem{LL10}
Lawler, Gregory F., and Vlada Limic. 
\newblock \emph{Random walk: a modern introduction.}
\newblock Vol. 123. Cambridge University Press, 2010.

\bibitem{LPW17}
Levin, David A., and Yuval Peres. 
\newblock \emph{Markov chains and mixing times.} Vol. 107. American Mathematical Soc., 2017.

\bibitem{LPS16}
Levine, Lionel, Wesley Pegden, and Charles K. Smart. 
\newblock ``Apollonian structure in the Abelian sandpile." 
\newblock \emph{Geometric and functional analysis} 26.1 (2016): 306-336.

\bibitem{LP10}
Levine, Lionel and Propp, James.
\newblock ``What is … a sandpile?'' 
\newblock \emph{Notices Amer. Math. Soc.} 57 (2010), no. 8, 976–979. 

\bibitem{LH02}
Lin, Chai-Yu, and Chin-Kun Hu. 
\newblock ``Renormalization-group approach to an Abelian sandpile model on planar lattices." 
\newblock \emph{Physical Review E} 66.2 (2002): 021307.


\bibitem{O48}
Otter, Richard. 
\newblock ``The number of trees." \emph{Annals of Mathematics} (1948): 583-599.

\bibitem{NOT17}
Nassouri, Estelle, Stanislas Ouaro, and Urbain Traore. 
\newblock ``Growing sandpile problem with Dirichlet and Fourier boundary conditions." 
\newblock \emph{Electronic Journal of Differential Equations} 2017.300 (2017): 1-19.

\bibitem{PS95}
Papoyan, Vl V., and R. R. Shcherbakov. 
\newblock ``Abelian sandpile model on the Husimi lattice of square plaquettes." 
\newblock \emph{Journal of Physics A: Mathematical and General} 28.21 (1995): 6099.

\bibitem{PS13}
Pegden, Wesley, and Charles K. Smart. 
\newblock ``Convergence of the Abelian sandpile."
\newblock \emph{Duke Mathematical Journal} 162.4 (2013): 627-642.

\bibitem{PR05}
Piroux, Geoffroy, and Philippe Ruelle. 
\newblock ``Boundary height fields in the Abelian sandpile model."
\newblock \emph{Journal of Physics A: Mathematical and General} 38.7 (2005): 1451.

\bibitem{P94}
Priezzhev, Vyatcheslav B. 
\newblock ``Structure of two-dimensional sandpile. I. Height probabilities." 
\newblock \emph{Journal of statistical physics} 74.5-6 (1994): 955-979.

\bibitem{SV09}
Schmidt, Klaus, and Evgeny Verbitskiy. 
\newblock ``Abelian sandpiles and the harmonic model."
\newblock \emph{Communications in Mathematical Physics} 292.3 (2009): 721.


\bibitem{S15}
Sokolov, Andrey, et al. 
\newblock ``Memory on multiple time-scales in an abelian sandpile.'' 
\newblock \emph{Physica A: Statistical Mechanics and its Applications} 428 (2015): 295-301.

\bibitem{S13}
Spitzer, Frank. 
\newblock \emph{Principles of random walk.} 
\newblock Vol. 34. Springer Science \& Business Media, 2013.

\bibitem{S04}
Steele, J. Michael. 
\newblock \emph{The Cauchy-Schwarz Master Class.} Cambridge University Press, 2004.

\bibitem{TV06}
Tao, Terence, and Van H. Vu. 
\newblock \emph{Additive combinatorics.} Vol. 105. Cambridge University Press, 2006.

\end{thebibliography}

\end{document}